\documentclass[final,3p,authoryear,12pt]{elsarticle}
\usepackage[round]{natbib}
\usepackage[colorinlistoftodos,textsize=footnotesize]{todonotes}%\usepackage[textsize=tiny]{todonotes}

\usepackage{amsfonts,amsmath,amssymb,lscape,float,graphicx,tabularx,url,times,color,natbib,afterpage,rotating,hyperref,xcolor,tikz,moresize,multirow,enumerate,colortbl,mathtools,bm,bbm}
\usepackage[utf8]{inputenc}
\usepackage{mathtools}
\usepackage[caption=true]{subfig}
\usepackage[ruled, vlined]{algorithm2e}
\usepackage{enumitem}
\usepackage{amsmath}
\usepackage{amssymb}
\usepackage{amsthm}
\mathtoolsset{showonlyrefs}
\allowdisplaybreaks
\usepackage{diagbox}

\newcounter{algsubstate}

\usepackage{setspace} % cmr cmss cmtt
\usepackage{geometry,float}
\definecolor{myblue}{rgb}{0.8,0.8,1}
\definecolor{myred}{rgb}{1,0.8,0.8}
\definecolor{mygreen}{rgb}{0.8,1,0.8}
\definecolor{mygrey}{rgb}{220,220,220}

\usepackage[bottom]{footmisc}
\usepackage[scaled]{helvet}

\DeclareMathAlphabet{\xcal}{OMS}{cmsy}{m}{n}

\definecolor{dbblue}{RGB}{10,65,155}				
\definecolor{dbred}{RGB}{215,0,50}					
\definecolor{blue}{RGB}{0,113.9850,188.9550} 			
\definecolor{red}{RGB}{216.7500,82.8750,24.9900} 		
\definecolor{green}{RGB}{118.8300,171.8700,47.9400} 	
\definecolor{grey}{RGB}{110,110,110}				
\definecolor{lgrey}{RGB}{210,210,210}				
\definecolor{c1}{RGB}{0,113.9850,188.9550}			
\definecolor{c2}{RGB}{216.7500,82.8750,24.9900}		
\definecolor{c3}{RGB}{236.8950,176.9700,31.8750}		
\definecolor{c4}{RGB}{125.9700,46.9200,141.7800}		
\definecolor{c5}{RGB}{118.8300,171.8700,47.9400}		
\definecolor{c6}{RGB}{76.7550,189.9750,237.9150}		
\definecolor{c7}{RGB}{161.9250,19.8900,46.9200}		
\definecolor{c14}{RGB}{0,102,102}					

\hypersetup{colorlinks,urlcolor=dbblue,citebordercolor=dbblue,linkbordercolor=dbblue,filecolor=dbblue,filebordercolor=dbblue,citecolor=dbblue,linkcolor=dbblue,colorlinks=true}

\newtheorem{definition}{Definition}
\newtheorem{lemma}{Lemma}
\newtheorem{remark}{Remark}
\newtheorem{proposition}{Proposition}

\newtheorem{example}{Example}

\defcitealias{UK7year:2013}{Griffiths et al., 2013}

\setcitestyle{citesep={,}}

\makeatletter
\def\ps@pprintTitle{%
  \let\@oddhead\@empty
  \let\@evenhead\@empty
  \def\@oddfoot{\reset@font\hfil\thepage\hfil}
  \let\@evenfoot\@oddfoot
}
\makeatother

\usepackage{xfrac,bigints}
\usepackage{bbold}

\newcommand{\BeginAppendix}{\appendix}
\newcommand{\EndAppendix}{}

\begin{document}

\begin{frontmatter}

\title{Optimal execution in a time-varying environment: well-posedness and price manipulation}

\author[sns]{Gianluca Palmari}
\ead{gianluca.palmari@sns.it}

\author[bologna,sns]{Fabrizio Lillo}
\ead{fabrizio.lillo@sns.it}

\author[imperial]{Zoltan Eisler}
\ead{eisler.zoltan@gmail.com}

\address[sns]{Classe di Scienze, Scuola Normale Superiore, Piazza dei Cavalieri 7, 56126 Pisa, Italy}
\address[bologna]{Department of Mathematics, University of Bologna, Piazza di Porta San Donato 5, Bologna 40126, Italy}
\address[imperial]{Department of Mathematics, Imperial College London, 180 Queen's Gate, South Kensington, London SW7 2RH, UK}

\begin{abstract}
We investigate the well posedness and the existence of price manipulation in an optimal execution problem within the Almgren-Chriss framework, where the temporary and permanent impact parameters vary deterministically over time. We present sufficient conditions for the existence of a unique solution and provide second-order conditions for the problem, with a particular focus on scenarios where impact parameters change monotonically over time. Additionally, we establish conditions to prevent transaction-triggered price manipulation in the optimal solution, i.e. the occurrence of buying and selling in the same trading program. Our findings are supported by numerical analyses that explore various regimes in simple parametric settings for the dynamics of impact parameters.
\end{abstract}


\end{frontmatter}

\bibliographystyle{abbrvnat}

% Shared body for the MOOR (informs4) and arXiv/SSRN wrappers.
% Edit content here so both versions stay identical.

\section{Introduction}\label{sec:Introduction}

Market frictions and trading costs significantly affect the investment performance of institutional investors. These include direct fees for market access and indirect costs arising from market microstructure effects, such as bid-ask spreads and price impact costs \cite{bouchaud2009markets, cont2014price}. Bid-ask spreads result from information asymmetry and inventory management by market makers \cite{glosten1985bid}. Impact costs are also often due to information asymmetry \cite{kyle1985continuous}. They mechanically arise due to the limited availability of liquidity, and their statistical properties are extensively studied \cite{lillo2003master, bouchaud2009markets, bucci2020co, zarinelli2015beyond}. Impact is often decomposed into two components. First, temporary impact reflects the short-term behavior of liquidity supply, and the immediate cost associated with a trade. Second, permanent impact reflects the long-term influence of a trade on the price, whose effect is accumulated over time.

Minimizing trading costs is an important topic in both the mathematical finance community and the financial industry \cite{bouchaud2009markets, taranto2018linear, obizhaeva2013optimal, almgren2001optimal, fruth2014optimal, fruth2019optimal}. As indicated by  \cite{kyle1985continuous}, for market participants it is optimal to split their trade into smaller orders. Formally such Optimal Execution Problems involve the interplay of two opposing constraints: ensuring the total quantity is traded within the allotted time interval, and navigating impact to minimize adverse price moves.

Numerous optimal execution models have been proposed and are actively utilized in the financial industry. The seminal paper of \cite{almgren2001optimal} introduced a price impact model, where permanent and temporary impact are linear functions of the traded volume and do not vary in time. Other approaches such as the propagator model \cite{bouchaud2003fluctuations}, account for the transient nature of impact, relaxing over time. \cite{obizhaeva2013optimal} delve into dynamic properties of supply and demand, and  \cite{fruth2014optimal} proposed time-varying models of impact based on Obizhaeva and Wang's initial model. 

A market impact model should not allow \emph{price manipulation}, sometimes called dynamic arbitrage. In simple terms, this means that by trading in and out of a position one should not generate a profit by exploiting price impact. \cite{alfonsi2012order} extended this notion by coining the term \emph{transaction-triggered price manipulation}, which occurs when, for example, the cost of buying (or selling) a given amount can be reduced by inserting intermediate trades of the opposite sign. By analyzing a model with linear transient and permanent impact components, they established necessary conditions for negative transaction costs.
In this context, one objective of our study is to derive conditions for well-posedness from both a mathematical and a financial perspective: mathematically, in the Hadamard sense (existence and uniqueness of a stable solution, see \cite{hadamard1902problemes}), and financially, in terms of the possible presence of transaction-triggered price manipulation (and, more generally, negative transaction costs). We call an Optimal Execution Problem \emph{strongly well-posed} if it is mathematically well-posed and its optimal liquidation strategy is monotonic (i.e., it does not contain intermediate buy trades). If it is mathematically well-posed but the optimal strategy involves both buy and sell trades, we call it \emph{weakly well-posed}.

\begin{figure}
        \centering
        \includegraphics[scale=0.76]{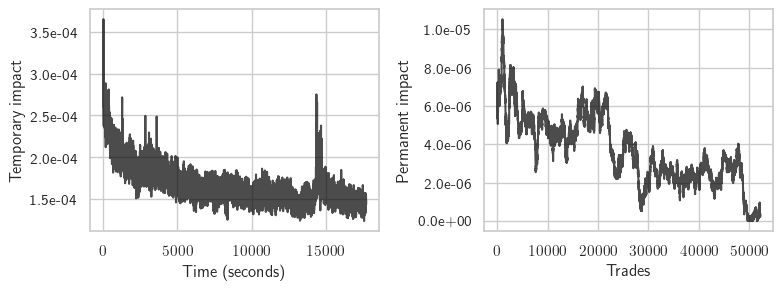}
        \caption{\textbf{Left:} Average intraday temporary impact dynamics for shares of Microsoft over June 2021. \textbf{Right:} Intraday permanent impact dynamics for shares of Microsoft on June 10, 2021.}
        \label{fig:impact_dynamics}
\end{figure}

It is very often assumed that impact parameters are constant over time. While convenient for tractability, this assumption overlooks the highly dynamic nature of liquidity: its availability fluctuates throughout the trading day, responds to the arrival of expected or unexpected news, and can even vanish endogenously during extreme events, as exemplified by the Flash Crash, see \cite{kirilenko2017flash}.

To motivate the need for incorporating dynamic impact parameters, Figure~\ref{fig:impact_dynamics} presents empirical estimations of the temporary (left panel) and permanent (right panel) impact coefficients for Microsoft shares during June 2021. The temporary impact is estimated following the methodology of \cite{cartea2016incorporating}, by traversing the limit order book to compute the price per share at different traded volumes relative to the prevailing mid-price, and subsequently performing a linear regression. The slope of this regression provides an estimate of the temporary impact at each time point. The permanent impact is estimated following the approach of \cite{campigli2022measuring}, which employs a score-driven (see \cite{blasques2015information}) modification of the  \cite{hasbrouck1991measuring} VAR model. From both panels, it is evident that impact parameters exhibit significant intraday variation. However, despite the stochastic fluctuations, the dynamics are largely centered around a deterministic trajectory with a pronounced J-shaped pattern. 

In the mathematical finance literature, several works have introduced stochastic
impact models for optimal execution (see, e.g.,
\cite{cartea2015optimal,gueant2016financial,barger2019optimal}), although they
typically focus on mean-reverting dynamics with constant mean. Empirical
evidence suggests that, while random fluctuations are present, the dominant
component of intraday impact behavior remains predominantly deterministic, with
stochastic deviations around it.
 
In this paper, we study optimal execution in settings where liquidity and market impact are deterministically time-dependent. Interestingly, to the best of our knowledge, this aspect has been neglected by the literature. Our primary focus is on the well-posedness of the optimal execution problem, with particular emphasis on establishing conditions for the existence and uniqueness of solutions. Additionally, we investigate the absence of potential price manipulation strategies. As we will detail below, if the speed of variation in liquidity is sufficiently high, manipulation opportunities may arise within the optimal strategies, or the problem may become ill-posed. %
Our main message is that time-variation can fundamentally change the qualitative behavior of optimal execution: if permanent impact decays too quickly relative to the prevailing level of temporary impact, then the optimal strategy may exhibit oscillations and, in extreme cases, the problem may become ill-posed or admit manipulation-type effects. More precisely, in continuous time we identify a quantitative admissibility condition requiring that the maximal decay rate of permanent impact be controlled by the minimal temporary impact over the trading horizon (up to the appropriate time-scale factor). In discrete time, this condition arises from a structural property of the impact matrix: we show that the problem is uniquely solvable and strongly well-posed—hence in particular free of transaction-triggered price manipulation—whenever the impact matrix is a $B$-matrix in the sense of \cite{pena2001class}. This yields an easily verifiable criterion directly in terms of the sampled impact coefficients, and it provides a clear restriction on admissible time dynamics for calibrated impact curves.

%Specifically, we study the continuous time scenario to identify weak conditions ensuring the existence and uniqueness of solutions. Additionally, we deepen the understanding of the permissible impact settings derived from the time-varying Almgren-Chriss model. Here, the variability in impacts challenges the existence of solutions to the optimal execution problem, potentially leading to scenarios where the problem does not conform to a typical minimization framework. In contrast, with constant impacts, the continuous time optimal execution problem, not only remains a minimization problem, but also guarantees a unique solution.

Our paper starts from a classical Almgren-Chriss setting in discrete time, assuming that the temporary and permanent impact coefficients are time-dependent. Section \ref{sec:problem} lays the groundwork for our analysis by presenting the problem setting, establishing the notation, and defining the concept of well-posedness. Building upon this foundation, the section then leverages discrete time settings and B-matrices theory to derive results in the continuous time limit on admissible regime dynamics, specifically for conditions ensuring strong well-posedness. Section \ref{secton3:CT} presents our analytical results for the continuous time setting. Using classical calculus of variations the problem is recast into a second order differential equation with Dirichlet boundary conditions. Leveraging the theory of such boundary problems and the Sturm-Picone Lemma, we provide sufficient conditions for the existence of a unique extremal point and second-order conditions guaranteeing that this point is a maximum. We also provide conditions for the positivity and monotonicity of the solution and therefore for the absence of transaction-triggered price manipulation. Section \ref{sec:numerical} provides a detailed numerical analysis of the properties of the optimal solution when temporary and permanent impact are exponentially time-varying\footnote{The Appendix also contains the case of power-law decaying impact coefficients.}. Numerical analysis confirms our theoretical results and shows that our sufficient conditions are rather binding. Finally in Section \ref{sec:conclusions}
we draw some conclusions. The paper contains several Appendices, where the proofs of our propositions are presented.

\section{The market impact model and the Optimal Execution Problem}\label{sec:problem}

In this Section, we introduce discrete time and continuous time market impact models and the associated Optimal Execution Problems. 
Throughout this paper, we extensively discuss the well-posedness of the Optimal Execution Problem, which is defined specifically as a Hadamard problem (see \cite{hadamard1902problemes}). In other words, a well-posed Optimal Execution Problem is characterized by the existence of a unique solution, which can be computed either analytically or numerically, and its numerical computation is stable. In the following we define when a problem is weakly or strongly well-posed, give sufficient conditions for each, and study their financial properties. 
%One can ask the question whether, given a set of parameters, Problem \eqref{problem:OTContinuousTime} is well-posed in the Hadamard sense \cite{hadamard1902problemes}. 
%We introduce two definitions of well-posedness.

\begin{definition}[Strongly well-posed]
\label{definition:strongly_well_posed}
An Optimal Execution Problem is \emph{strongly well-posed} if it is well-posed
in the Hadamard sense and its optimal liquidation strategy is monotone (i.e.,
it does not contain intermediate trades of the opposite sign). In discrete time
this means that an optimal sell program satisfies $\xi_i\le 0$ for all
$i=1,\dots,N$, and in continuous time it means $\dot\zeta_t\le 0$ for all
$t\in[0,T]$.
Note that Proposition~\ref{proposition:NoPTPimpliesNOPM} implies that such a
problem admits no negative expected transaction costs.
\end{definition}

\begin{definition}[Weakly well-posed]
We say that an Optimal Execution Problem is \emph{weakly well-posed} if it is
well-posed in the Hadamard sense, but its optimal liquidation strategy is not
monotone (i.e., it involves both buy and sell trades).
\end{definition}

Proposition \ref{proposition:NoPTPimpliesNOPM} plays a pivotal role in practical applications by enabling the proof of non-arbitrage results. It allows us to establish the strong well-posedness of a model by demonstrating that the price impact model does not admit transaction-triggered price manipulation. 

\subsection{The discrete time framework}\label{secton2:DT}

%We begin with the discrete time impact model, which will help building an intuition for the continuous time market impact model. In our stock price model a trader has the capacity to impact asset prices. In the absence of the trader's activity, asset prices follow a random walk representing the aggregated actions of all other market participants.

An agent fixes a finite time horizon of trading $T>0$, and splits the interval
$[0,T]$ into $N$ equal sub-intervals. We set
\[
t_i=\frac{iT}{N},\qquad i=0,\dots,N,
\]
so that the trading periods are $[t_{i-1},t_i)$ for $i=1,\dots,N$. Throughout
the paper, we adopt the convention that \(Q>0\) denotes the size
of the initial position. We mainly focus on the liquidation of a long
position, i.e., a \emph{sell program} that reduces an initial inventory
of \(Q\) shares to zero. Monotone liquidation corresponds to non-positive
trades; however, unless otherwise stated we allow intermediate buy trades in
order to analyze transaction-triggered price manipulation.
Buy programs can be treated symmetrically by reversing the sign of all
trades and of the initial inventory.

\begin{definition}[Trading strategy in discrete time]
\label{definition:strategy}
Let \(Q>0\) denote the size of an initial long inventory.
A trading strategy in discrete time is a sequence
\(\xi_1,\dots,\xi_N \in \mathbb{R}\) of trades on the sub-intervals
\([t_{i-1},t_i)\). The associated inventory process is
\[
\zeta_0 = Q,\qquad \zeta_i = \zeta_{i-1} + \xi_i,\quad i=1,\dots,N.
\]
A \emph{sell program} (full liquidation) satisfies \(\zeta_N=0\), i.e.
\[
\sum_{i=1}^{N} \xi_i = -\,Q,
\]
with \(\xi_i \in \mathbb{R}\) for all \(i\)\footnote{A \emph{buy program} is obtained
symmetrically by reversing the sign of all trades and of the inventory.}.
\end{definition}

%\begin{definition}
%\label{definition:strategy0}
%We consider an admissible strategy $\xi$ for $Q\in \mathbb{R}$. A negative transaction costs strategy, is an admissible strategy such that $\mathcal{C}(\xi) < 0$. In the sequel we will consider such a control to be a price manipulation strategy.
%\end{definition}

\begin{definition}
\label{definition:previsible_strategy}
A strategy is previsible if $\xi_i$ is $\mathcal{F}_{i-1}$-measurable, where
$\mathcal{F}_{i-1}$ is the information set available to the trader up to the
end of the $(i-1)$st interval (i.e., before time $t_{i-1}$). We define
$\mathbb{F}=\{\mathcal{F}_i\}_{i=0}^{N}$ as the filtration induced by the noise
process $\{\varepsilon_i\}_{i=1}^{N}$, where
$\varepsilon_i\overset{\text{ind}}{\sim}\mathcal{N}(0,\sigma^{2}_i)$.
Note that we do not impose any sign property on the trading process.
\end{definition}

\begin{definition}
\label{definition:admissible_strategy}
For a fixed inventory size \(Q>0\) we define the space of admissible
(liquidation) strategies by
\[
\Xi_N(Q)
= \Big\{\xi\in \mathbb{R}^N:\ \sum_{i=1}^{N}\xi_{i} = -\,Q,\
\xi_i\in\mathcal{F}_{i-1},\ \exists M > 0\ \text{s.t. } |\xi_{i}| < M,\
\forall i=1,\dots, N\Big\}.
\]
Admissible strategies are \(\mathbb{F}\)-adapted and bounded.
\end{definition}

\begin{definition}
\label{definition:static_admissible_strategy}
For a fixed inventory size \(Q>0\) we define the space of static admissible
strategies by
\[
\Xi_{N,0}(Q)
= \Big\{\xi\in \mathbb{R}^N:\ \sum_{i=1}^{N}\xi_{i} = -\,Q,\
\xi_i\in\mathcal{F}_{0},\ \exists M > 0\ \text{s.t. } |\xi_{i}| < M,\
\forall i=1,\dots, N\Big\}.
\]
Static strategies are chosen at time \(0\) and are bounded.
\end{definition}

\begin{definition}[Time-varying Almgren-Chriss model in discrete time]
The  \cite{almgren2001optimal} model encodes the agent's impact on the market price $S$ as  
\begin{eqnarray}
\label{eq:S}
    S_{k}=S_{k-1}+\theta_{k} \xi_{k} + \varepsilon_{k},
\end{eqnarray}
while the price experienced by the trader is defined by 
\begin{eqnarray}
    \label{eq:S_tilde}
    \tilde{S}_{k} = S_{k-1} + \eta_{k} \xi_{k}.
\end{eqnarray}
\end{definition}

The Almgren-Chriss model includes two types of market impact. First, the process $\{\theta_k\}_{k\ge0}$ is $\mathcal{F}_{0}$-measurable and positive, and represents the \emph{permanent} impact. When the agent trades  $\xi_k$ shares at time $t_{k}$, she permanently impacts the market price by the value $\theta_k\xi_k$. Second, the model also considers an additional time-varying impact, called \emph{temporary} impact and denoted by the $\mathcal{F}_{0}$-measurable and positive process $\{\eta_k\}_{k\ge0}$. This impact completely dissipates after the trade, but affects the transaction costs.

%\begin{definition}
%\label{definition:model}
%We define a price impact model as $\tilde S_k(\xi_1, \dots, \xi_k)$ which maps trades up to the $k$`th interval to the tradeable price $\tilde S_k$ of the $k$'th trade defined by Eq. \eqref{eq:S_tilde}. A price impact model is completely characterized by the price equation Eq. \eqref{eq:S} and the effective price equation \eqref{eq:S_tilde}. Equivalently a price impact model is completely characterized by the $\mathcal{F}_0$-measurable impact processes $\theta$ and $\eta$, and the noise process $\varepsilon$.
%\end{definition}

%\begin{definition}

The implementation shortfall for a discrete time strategy is defined as
the difference between the initial value $Q S_0$ of the portfolio at time
zero and the cash generated by the execution:
\begin{eqnarray}
\label{eq:C}
    \mathrm{IS}(\xi) = Q S_0 + \sum_{i=1}^{N}\xi_{{i}}\tilde{S}_{{i}}.
\end{eqnarray}
%\end{definition}

%\begin{definition}
The expected cost of trading is defined as
\begin{eqnarray}
\label{eq:DT}
    \mathcal{C}(\xi) = \mathbb{E}_{0}\left[\mathrm{IS}(\xi)\right].
\end{eqnarray}
Here $\mathbb{E}_{0}[\cdot]=\mathbb{E}[\cdot\,|\,\mathcal{F}_{0}]$ denotes
conditional expectation given the information available at time $0$.

and in turn, the Optimal Execution Problem is formulated as minimizing the cost, via
\begin{eqnarray}
\label{problem:DT}
   \xi^{\star} = \arg \min_{\xi\in\Xi_N(Q)}\mathcal{C}(\xi).
\end{eqnarray}
%\end{definition}

The objective of this paper is to analyze the solutions of Problem \eqref{problem:DT}, its continuous time counterpart (Problem \eqref{problem:OTContinuousTime} below), and to establish conditions for the existence and uniqueness. In order to do so, we will require two more fundamental definitions.

\begin{definition}
\label{definition:roundtrip}
Let us consider an admissible strategy $\xi$ for $Q > 0$. A discrete time price manipulation strategy is an admissible strategy such that $\mathcal{C}(\xi) < 0$. This is different from the definition of \cite{huberman2004price}, who considered the case of $Q=0$, where both buying and selling are permitted.
\end{definition}

\begin{definition}
\label{definition:TTPM}
A discrete-time model admits transaction-triggered price manipulation in the sense of
\cite{alfonsi2010optimal} if there exists an inventory size $Q>0$ and a
corresponding admissible sell strategy $\tilde{\xi}\in\Xi_N(Q)$ such that
\[
\mathcal{C}(\tilde{\xi}) <
\inf\Big\{ \mathcal{C}(\xi)\ \big|\ \xi\in\Xi_N(Q),\ \xi_i \le 0\ \forall i\Big\}.
\]
In words, the cost of a sell program can be reduced by inserting intermediate
buy trades.
\end{definition}

\medskip 
We define the impact encoding matrix \(A\) as follows 
\begin{equation}
\label{eq:encoding_matrix}
    A = \begin{bmatrix}
                2\eta_{1} & \theta_{1} & \dots & \dots & \theta_{1}\\
                \theta_{1} & 2\eta_{2} & \theta_{2} & \dots & \theta_{2} \\
                \dots & \theta_{2} & 2\eta_{3} & \dots & \dots \\
                \dots & \dots & \dots & \dots & \theta_{N-1} \\
                \theta_{1} & \theta_{2} & \dots & \theta_{N-1} & 2\eta_{N} 
                \end{bmatrix},
\end{equation}
where \(\{\theta_k\}_{k \ge 0}\) and \(\{\eta_{k}\}_{k \ge 0}\) are respectively the permanent and temporary (\(\mathcal{F}_0\)-measurable) impacts. 

The proofs of this and all subsequent propositions are presented in \ref{app:proofs}.

\begin{proposition}
\label{proposition:CostFunctionDT}
Assume the encoding matrix \(A\) defined by Eq. \ref{eq:encoding_matrix}
is positive semi-definite. Then the Optimal Execution Problem may be written as 
\begin{align}
    \xi^{\star} = \underset{\xi \in \Xi_{N,0}(Q)}{\arg \min}
    \left\{ \frac{1}{2}\xi^{T}A\xi\right\}\,,
\end{align}
with $\Xi_{N,0}(Q)$ defined in \ref{definition:static_admissible_strategy}.
\end{proposition}
%\begin{remark}
Notice that in the calculation of the optimal strategy the matrix $A$ encodes
the impacts. Notably, we observe that the trader incurs a cost proportional to
each of the temporary impacts once, and a cost proportional to the $i$th
permanent impact $N-i$ times. Moreover, Problem \eqref{problem:DT} is ``dynamic''
in the sense that we optimize with respect to the admissible set of strategies
$\Xi_{N}(Q)$. Those strategies are such that $\xi_i\in\mathcal{F}_{i-1}$.
Proposition \ref{proposition:CostFunctionDT} shows that the \emph{optimal}
solution $\xi^{\star}$ is ``static'', $\xi^{\star} \in \mathcal{F}_0$, i.e.\ it
is sufficient to minimize over the much smaller set $\Xi_{N,0}(Q)$.
%\end{remark}

Since the optimization problem is quadratic, it is immediate to see that the following Proposition holds.

\begin{proposition}
\label{proposition:DTProblemOptimalsolution}
If $A$ is symmetric positive definite (SPD), then the Optimal Execution
Problem has a unique solution $\xi^{\star}$, and 
\begin{eqnarray}
    \xi^{\star}= -\,\frac{Q}{\pmb{1}^{T}A^{-1}\pmb{1}}\, A^{-1}\pmb{1}.
\end{eqnarray}
where $\pmb{1}\in\mathbb{R}^{N}$ denotes the vector with all components equal
to $1$.
\end{proposition}

Before diving into the formal properties of our optimal execution model, we present two examples that elucidate the mathematical and financial motivations for studying optimal execution in a time-varying setting.

\begin{example}[Constant impact]
\label{example:constant_impact}
Assume time-independent constant impact parameters $\theta$ and $\eta$.
%By Proposition \ref{proposition:CostFunctionDT}, the Optimal Execution Problem becomes 
%\begin{align}
%\label{problem:example_constant_impact}
%    \xi \in \underset{\xi \in Q}{\arg \min} \left\{ \frac{1}{2}\xi^{T}A\xi\right\}, && A = \begin{bmatrix}
 %               2\eta & \theta & \dots & \dots & \theta\\
 %               \theta & 2\eta & \theta & \dots & \theta \\
 %               \dots & \theta & 2\eta & \dots & \dots \\
 %               \dots & \dots & \dots & \dots & \theta \\
 %               \theta & \theta & \dots & \theta & 2\eta 
 %               %\end{bmatrix},
%\end{align}
Note that the spectrum of \(A\) can be computed and is equal to \(\mathrm{Sp}(A) = \{2\eta-\theta,2\eta+(N-1)\theta\}\). The first eigenvalue has algebraic multiplicity $N-1$, % with the associated eigenvectors $\left\{(1,-\frac{1}{N-1},\dots,\frac{1}{N-1}),\dots,(-\frac{1}{N-1},\dots,\frac{1}{N-1},1)\right\}$. 
while the second eigenvalue has algebraic multiplicity one. Thus, if $\eta > \frac{\theta}{2}$ the optimal solution is the TWAP strategy $\xi^{\star}_{{i}}=-\frac{Q}{N}$ \cite{almgren2001optimal}. This is because $A$ is SPD, which makes the problem convex. On the contrary if $\eta < \frac{\theta}{2}$, then the worst strategy is the TWAP, because $-A$ is SPD, and thus the problem is concave.
\end{example}

\begin{example}[Time-varying impact]
\label{example:_TV_impact}
%The previous example illustrated a necessary and sufficient condition on the impact parameters for a well-posed optimal execution problem W
One might ask whether the above condition $\eta_{{i}}>\frac{\theta_{{i}}}{2}$ is an insufficient, sufficient, or necessary condition to ensure a well-posed minimization problem. It turns out, none of these is true in general.

Consider the matrix
\begin{eqnarray}
    A = \begin{bmatrix}
                2\eta_{1} & \theta_{1} & \theta_{1} \\
                \theta_{1} & 2\eta_{2} & \theta_{2}\\
                \theta_{1} & \theta_{2} & 2\eta_{3}
            \end{bmatrix},
\end{eqnarray}
where $\theta=(1,3,\theta_{3})$ and $\eta = (2,2,\eta_{3})$. The two first minors
of $A$ are positive, implying that $\det(A) > 0$ if and only if $A$ is
symmetric positive definite (SPD). It is noteworthy that $\det(A) < 0$ if and
only if $\eta_{3} < \frac{17}{15}$. Therefore, if $\eta_{3} < \frac{17}{15}$
and $\theta_{3} < \frac{34}{15}$, the problem is ill-posed. For instance, when
$\theta_{3}=2$ and $\eta_{3} \in (1,\frac{19}{15})$, the Optimal Execution
Problem becomes a saddle point problem, resulting in an unbounded cost function
over $Q$. Note that $\theta_{3}$ does not have any influence on the
optimization. It is chosen to be smaller than $\frac{34}{15}$ only to be
consistent with the initial assumption $\eta_{{i}}>\frac{\theta_{{i}}}{2}$ for
$i\in\{1,2,3\}$.  

An alternative condition, $\eta_{{i}}>(N-1)\frac{\theta_i}{2}$, ensures that $A$ is diagonally dominant, hence SPD. This establishes strict convexity, guaranteeing the existence of an optimal execution strategy without generating negative transaction costs. However, this condition becomes increasingly restrictive for large $N$, and especially as $N$ approaches infinity. %It is therefore crucial to identify less restrictive conditions on $A$ that still ensure we can obtain an optimal solution.
\end{example}

%\begin{remark}
It is possible to show -- using the Lagrange optimality conditions -- that when
the permanent impact $\theta$ is constant, the condition
$\eta_{i}>\frac{\theta}{2}$ for all $i\in \{1,\dots,N\}$ guarantees that the
Optimal Execution Problem is well-posed and the optimal solution does not
change sign during the execution, i.e. the problem is strongly well-posed.
However, the previous example shows that $\eta_{i}>\frac{\theta_{i}}{2}$ is not
sufficient for having a well-posed problem when permanent impact is
time-varying. We therefore need to impose other conditions.

%\end{remark}
Using the property of diagonally dominant matrices, it is straightforward to prove the following.

\begin{proposition}
\label{proposition:diagdom}
If 
\begin{eqnarray}
    2\eta_{i} > \sum_{j=1}^{i-1}\theta_{j}+(N-i)\theta_{i},
\end{eqnarray}
then the Optimal Execution Problem is weakly well-posed.
\end{proposition}

Proposition \ref{proposition:diagdom} provides a sufficient condition for Hadamard well-posedness by rendering the problem strictly convex. However, convexity alone does not offer insights into the nature of the optimal solution. Despite $A$ being SPD, the optimal solution might still involve both buys and sells. To address such scenarios, we introduce the B-matrix class (see \cite{pena2001class}), which proves instrumental in resolving these questions.

\begin{definition}
    Let $A=(a_{i,k})_{1\le i,k\le N}$ be a real square matrix. A is a B-matrix when it has positive row sums and all of its off-diagonal elements are bounded above by the corresponding row means, i.e. for all $i=1,\dots,N$
    \begin{eqnarray}
    \label{inequality:Bmatdef}
        \sum_{k=1}^{N}a_{i,k}>0 \text{, and } \frac{1}{N}\left(\sum_{k=1}^{N}a_{i,k}\right)>a_{i,j} \text{, }\forall j \neq i.
    \end{eqnarray}
\end{definition}

Both the B-matrix condition and diagonal dominance imply that the (symmetric) impact matrix \(A\) is positive semidefinite (in fact, under the B-matrix condition it is positive definite; see Lemma~\ref{lemma:BmatrixSPD}). However, the classes of B-matrices and diagonally dominant matrices are not comparable: neither class is contained in the other. For instance, in the case \(N=3\), take \(\eta_1=\eta_2=2\), \(\eta_3=1\), \(\theta_1=1.6\), and \(\theta_2=0.1\). Then \(A\) is diagonally dominant, but it fails the B-matrix inequalities (in particular, in the third row the row mean is smaller than \(\theta_1\)). Conversely, still for \(N=3\), take \(\eta_1=1\), \(\eta_2=0.6\), \(\eta_3=1\), and \(\theta_1=\theta_2=1\). Then \(A\) satisfies the B-matrix condition (cf.\ Proposition~\ref{proposition:Bmatrixequivalence}), but it is not diagonally dominant (the second row violates diagonal dominance).

%When $A$ is a B-matrix the execution problem is well-posed as shown by the following proposition. 

\begin{proposition}
\label{proposition:BmatrixNoTTPM}
    If $A$ is a B-matrix, the corresponding Optimal Execution Problem has a unique solution and is strongly well-posed.
\end{proposition}

The B-matrix condition is easy to verify and, by
Proposition~\ref{proposition:BmatrixNoTTPM}, it ensures strong well-posedness
of the discrete-time Optimal Execution Problem (existence, uniqueness and
absence of transaction-triggered price manipulation). In particular, it
imposes a restriction not only on the level of permanent impact, but also
on its rate of decrease relative to temporary impact.
\begin{proposition}
\label{proposition:_ct_limit_for_dt_model}
Assume the permanent impact coefficients $(\theta_i)_{i=1}^N$ form a
non-increasing sequence,
\[
\theta_1 \ge \theta_2 \ge \dots \ge \theta_N>0.
\]
Let
\[
\underline\eta = \min_{1\le i\le N} \eta_i.
\]
If
\begin{equation}
\label{inequality:BmatRestrictive}
2\,\underline\eta \;\ge\; N\theta_1 - (N-1)\theta_N,
\end{equation}
then the impact matrix $A$ defined in \eqref{eq:encoding_matrix} is a
$B$-matrix.
\end{proposition}
Proposition~\ref{proposition:_ct_limit_for_dt_model} provides a simple
sufficient condition to check the B-matrix property directly on the discrete
impact coefficients. When the latter are obtained by filtering or estimating smooth
continuous-time curves, this yields a first constraint on the admissible
dynamics of $(\theta_t,\eta_t)$.

When the discrete coefficients are obtained by sampling smooth curves, an
economically meaningful high-frequency limit requires the temporary-impact
coefficient to scale as $1/\Delta t$, since $\xi_i$ represents a trade \emph{size}
executed over the interval $[t_{i-1},t_i]$. Under this scaling, imposing the
$B$-matrix property on the sequence of matrices associated with refining time
grids yields the following continuous-time admissibility condition.

\begin{proposition}
\label{prop:CT_from_Bmatrix}
Let $T>0$. Let $\theta:[0,T]\to(0,\infty)$ and $\eta:[0,T]\to(0,\infty)$ be
continuous, and assume that $\theta$ is strictly decreasing. For each
$N\in\mathbb{N}$, set
$\Delta t=T/N$ and $t_i=i\Delta t$ for $i=0,\dots,N$.
Define sampled coefficients on $[t_{i-1},t_i]$ by
\[
\theta_i^{(N)}=\theta(t_{i-1}),\qquad
\eta_i^{(N)}=\frac{\eta(t_{i-1})}{\Delta t},
\qquad i=1,\dots,N.
\]
Let $A^{(N)}$ be the impact matrix \eqref{eq:encoding_matrix} built from
$\bigl(\theta_i^{(N)},\eta_i^{(N)}\bigr)_{i=1}^N$ (i.e.\ $a^{(N)}_{ii}=2\eta_i^{(N)}$
and $a^{(N)}_{ij}=\theta_{\min(i,j)}^{(N)}$ for $i\neq j$).
Assume that, for every $N$, the matrix $A^{(N)}$ is a $B$-matrix.

Then, for every $t\in[0,T]$,
\begin{equation}
\label{eq:CT_pointwise_admissibility}
\frac{2}{T}\eta(t)
+\frac{1}{T}\int_{0}^{t}\theta(u)\,du
+\Bigl(1-\frac{t}{T}\Bigr)\theta(t)
\;\ge\;\theta(0).
\end{equation}
If, in addition, $\theta\in C^{1}([0,T])$, then
\eqref{eq:CT_pointwise_admissibility} can be rewritten as
\begin{equation}
\label{eq:CT_weighted_slope}
2\,\eta(t)\;\ge\;\int_{0}^{t}(T-s)\,(-\dot\theta(s))\,ds,
\qquad \forall t\in[0,T].
\end{equation}
In particular, at terminal time,
\begin{equation}
\label{eq:CT_terminal_condition}
2\,\eta(T)\;\ge\;\int_{0}^{T}(T-s)\,(-\dot\theta(s))\,ds
= T\theta(0)-\int_{0}^{T}\theta(u)\,du.
\end{equation}
\end{proposition}

\begin{remark}
\label{remark:CT_from_Bmatrix_nonmonotone}
The monotonicity of $\theta$ is only used to identify the largest
off-diagonal entry in each row of $A^{(N)}$. If $\theta$ is merely continuous,
the same argument yields \eqref{eq:CT_pointwise_admissibility} with the
right-hand side replaced by the running maximum
$\theta^{\ast}(t):=\max_{u\in[0,t]}\theta(u)$. In the strictly decreasing case,
$\theta^{\ast}(t)=\theta(0)$ for all $t\in[0,T]$.
\end{remark}

Proposition~\ref{prop:CT_from_Bmatrix} yields a continuous-time
\emph{admissibility} condition implied by discrete-time strong
well-posedness (via the $B$-matrix property on refined grids).
Written as
\[
2\,\eta(t)\;\ge\;\int_{0}^{t}(T-s)\,(-\dot\theta(s))\,ds,
\]
it highlights the interplay between the two impact channels.

The quantity $-\dot\theta(s)$ is the \emph{instantaneous rate at which permanent
impact decreases}. Faster decrease increases incentives to insert buy--sell
round-trips: one may buy to move the price and later unwind after the
permanent component has partially dissipated. The weight $(T-s)$ reflects
that decrease occurring earlier is more destabilizing, because it affects a
larger fraction of the remaining horizon.

Temporary impact $\eta(t)$ acts as an instantaneous trading friction that
penalizes aggressive intermediate trades. The inequality therefore says
that, to preserve strong well-posedness (and in particular to rule out
transaction-triggered price manipulation),
temporary impact must be sufficiently large at each time $t$ to dominate the
remaining-horizon-weighted cumulative decay of permanent impact up to $t$.

The terminal condition
\[
2\,\eta(T)\;\ge\;T\theta(0)-\int_{0}^{T}\theta(u)\,du
\]
shows that the most stringent constraint is driven by the \emph{average permanent impact} over the horizon relative to the initial level $\theta(0)$:
substantial decrease of permanent impact must be compensated by sufficiently
strong temporary trading costs.

\medskip
We return to the continuous-time constraint \eqref{eq:CT_weighted_slope} in
Section~\ref{secton3:CT}. In particular, in
Subsection~\ref{subsection:CT_second_order} we show that it implies the global
second-order inequality \eqref{inequality:localminimumCalculusofVariations_ineq},
and we compare this sampling-induced admissibility with the purely
continuous-time criterion of Proposition~\ref{proposition:CT_is_well_defined}.

\subsection{The continuous time framework}

%\begin{definition}[Trading strategy in continuous time]
We now turn to the main objective of this paper, namely the determination of
well-posedness of the Optimal Execution Problem with time-varying impact
parameters in continuous time. Consider trading over the interval $t\in[0,T]$.
Throughout, we restrict to deterministic strategies, represented by an
inventory trajectory $\zeta:[0,T]\to\mathbb{R}$ and its trading rate
$\dot\zeta$.
%\end{definition}
\begin{definition}
\label{def:liquidation}
In continuous time, for a fixed initial inventory of $Q>0$ shares to
liquidate, we define the space of admissible strategies as
\[
\Xi(Q)
= \Big\{\zeta\in\mathcal{C}^{2}([0,T]):\ \zeta_{0}=Q,\ \zeta_{T}=0\Big\}.
\]
Admissible strategies may include both selling periods
($\dot\zeta_t<0$) and buying periods ($\dot\zeta_t>0$).
\end{definition}

To streamline notation, we take \(Q>0\) to denote an initial long position to be
liquidated over \([0,T]\) in both the discrete- and continuous-time frameworks.
In discrete time, selling corresponds to trades \(\xi_i<0\), while in continuous
time it corresponds to a negative trading rate \(\dot\zeta_t<0\). Admissible
strategies may include intermediate buying periods (\(\xi_i>0\) or
\(\dot\zeta_t>0\)), provided the boundary conditions \(\zeta_0=Q\) and
\(\zeta_T=0\) are satisfied.

%\begin{definition}[Time-varying Almgren-Chriss model in continuous time]
Following \cite{almgren2001optimal} we assume that the dynamics of the efficient price $S_t$ and cash process $X_t$ are given by
\begin{align}
        \label{equation:_CT_cash_process}
        dS_t &= \theta_t \dot \zeta_t dt+\sigma dW_t, \\
        \label{equation:_CT_cash_process_2}
        dX_t &= -\dot \zeta_t(S_t+\eta_t \dot \zeta_t)dt,
\end{align}
with permanent and temporary impact $\theta_t$ and $\eta_t$, respectively, and
volatility $\sigma \in \mathbb{R}_{+}$. We assume that $\theta$ and $\eta$ are
deterministic and continuously differentiable on $[0,T]$.
%\end{definition}

The connection of the continuous time model with the discrete time one is provided by the following proposition.

\begin{proposition}
\label{proposition:EMdiscretization}
The discrete-time model of Eqs.~\eqref{eq:S} and \eqref{eq:S_tilde} is obtained
as the Euler--Maruyama discretization of the continuous-time model
\eqref{equation:_CT_cash_process}--\eqref{equation:_CT_cash_process_2}.
\end{proposition}

Using integration by parts (It\^{o}'s formula) and assuming $X_0=0$, we get the
total cash generated by the selling procedure:
\begin{eqnarray}
    X_T(\zeta)
    =QS_0+\int_0^T(\theta_t \zeta_t \dot \zeta_t-\eta_t\dot \zeta_t^2)dt
    +\sigma\int_0^T \zeta_t\,dW_t.
\end{eqnarray}
The execution cost is defined as the initial value of the position minus the expected terminal cash. To align with the literature, since we consider a selling procedure, we consider the maximization problem that focuses on maximizing the terminal cash in excess. 
Define the cash functional 
\begin{eqnarray}
    \label{eq:cash}
    \mathcal{J}(\zeta)
    = \mathbb{E}_{0}\left[X_T(\zeta)-Q S_0\right]
    = \int_0^T \left(\theta_t \dot \zeta_t \zeta_t - \eta_t \dot \zeta_{t}^2\right) dt,
\end{eqnarray}
(the stochastic integral has zero expectation since $\zeta$ is deterministic), and obtain the continuous-time
Optimal Execution Problem formulation
\begin{eqnarray}
\label{problem:OTContinuousTime}
    \zeta^{\star} = \arg \max_{\zeta\in\Xi(Q)} \mathcal{J}(\zeta).
\end{eqnarray}

This convention is consistent with the one in discrete time.
Equation \eqref{problem:OTContinuousTime} maximizes the expected cash in
excess of the initial mark-to-market value, whereas its discrete-time
counterpart \eqref{problem:DT} minimizes expected trading costs. These
dual formulations are standard in the literature.

The maximization of the cash functional in the r.h.s. of
Eq.~\eqref{eq:cash} can be obtained by using the Euler–Lagrange equation,
see \cite{gueant2016financial}.
\begin{definition}[Extremal point]
\label{def:extremal_point}
An extremal point of $\mathcal{J}$ is a function $\zeta\in\Xi(Q)$ that satisfies
the boundary-value problem
\begin{align}
    \label{eq:extremal_point_BVP}
    \begin{cases}
        2\eta_t \ddot \zeta_t+2\dot \eta_t\dot \zeta_t-\dot \theta_t\zeta_t&=0, \\
        \zeta_{t=0}&=Q,\\
        \zeta_{t=T}&=0.
    \end{cases}
\end{align}
\end{definition}

Finally, we recall the definition of transaction-triggered price manipulation given by \cite{alfonsi2012order}.

\begin{definition}
\label{definition:TTPMcont}
A continuous-time model admits transaction-triggered price manipulation if
there exists an inventory size $Q>0$ and a corresponding admissible sell
strategy $\tilde \zeta$ such that
\[
\mathcal{J}(\tilde{\zeta}) >
\sup \Big\{ \mathcal{J}(\zeta)\ \big|\ \zeta\in\Xi(Q),\ \dot \zeta_t \le 0\
\forall t\in[0,T]\Big\}.
\]
That is, the expected cash from a liquidation can be increased by inserting
intermediate buy trades (periods with $\dot\zeta_t>0$).
\end{definition}

%and a theorem, due to \cite{alfonsi2010optimal} connecting the two types of manipulation.

\begin{proposition}
    \label{proposition:NoPTPimpliesNOPM}
     If a price impact model does not admit transaction-triggered price manipulation, then it does not admit arbitrage in the sense that there is no admissible strategy that generates negative transaction costs \cite{alfonsi2010optimal}.
\end{proposition}

%

%Let us clarify the difference between an Optimal Execution Model and an Optimal Execution Problem. The Optimal Execution Model refers to the general formulation of optimal execution derived from the impact model, specifically using Almgren-Chriss with time-varying impact in our framework. The Optimal Execution Problem is the mathematical problem defined by the Optimal Execution Model, understood as an optimization problem in the Hadamard sense. Understanding the strong or weak well-posedness of the Optimal Execution Problem presents a significant mathematical challenge. The qualitative results derived from this analysis are important for the practical applicability of the model.

\section{Well posedness of the continuous time Optimal Execution Problem}
\label{secton3:CT}

This Section considers the continuous time version of the time varying Almgren-Chriss optimal execution model. When impacts are constant, the derivation of the optimal solution is straightforward. We start by providing a condition for the existence and uniqueness of solutions in the time-varying case which is rather more complex. This is illustrated by some compelling examples. Then we delve deeper by formulating second-order conditions, and analyzing potential solution dynamics via numerical examples that reflect real-world financial markets. %We will revert to a discrete time formulation in Section \ref{secton2:DT}.

\subsection{Existence and uniqueness of the solution}
To begin, let us introduce a sufficient condition for the existence and uniqueness of the solution of the optimization problem.
\begin{proposition}
\label{proposition:ExistenceandUniquenessCT}
Let $f: (t,u_1,u_2)\in [0,T]\times\mathbb{R}^{2}\mapsto -\frac{\dot \eta_{t}}{\eta_{t}}u_{2} + \frac{\dot \theta_{t}}{2\eta_{t}}u_{1}\in \mathbb{R}$. 
If $f(t,\cdot)$ is continuous, uniformly Lipschitz, $\left \{\frac{\partial f(t,\cdot)}{\partial u_{i}}\right \}_{i=1,2}$ are continuous, and if there exists $\lambda \in (0,1)$ such that
       \begin{align}
       \label{existence_and_uniqueness_inequalities}
        \max\left\{\frac{1}{2}\int_{0}^{T}\left| \frac{\dot \theta_{t}}{\eta_{t}}\right| dt , \int_{0}^{T}\left| \frac{\dot \eta_{t}}{\eta_{t}}\right|dt \right\}&\le \log(1+2\lambda),
        \end{align}
then the Optimal Execution Problem \eqref{problem:OTContinuousTime} has a unique solution.
\end{proposition}

\subsection{Motivating Examples}

In order to motivate the study of the Optimal Execution Problem in a time-varying context, let us look at a few simple examples.

\begin{example}[Constant impact]
\label{example:_constant_impact_CT}

In the original Almgren-Chriss model the impact parameters are constant, i.e., $\theta_{t}=\theta>0$ and $\eta_{t}=\eta>0$. Consequently, $\dot{\eta}_t=0$ and $\dot{\theta}_t=0$. The Euler-Lagrange equations are
%This configuration aligns with the Almgren-Chriss problem, transforming the optimization problem into:
%\begin{eqnarray}
%    \zeta^{\star}\in \arg \max_{\zeta\in \Xi} \int_0^T(\theta \dot \zeta_{t} \zeta_t-\eta \dot \zeta_{t}^2)dt
%\end{eqnarray}

%As a calculus of variation problem, applying the Euler-Lagrange formula yields the following second-order ordinary differential equation:

\begin{align}
    \begin{cases}
    \ddot \zeta_{t} &= 0,\\
    \zeta_{t=0}&=Q,\\
    \zeta_{t=T}&=0.
\end{cases}
\end{align}
The solution is $\zeta_t=Q(1-\frac{t}{T})$ i.e. the program trades at a constant rate $ \dot \zeta_{t} = -\frac{Q}{T}$. This is called the TWAP strategy. The solution is unique, because inequality \eqref{existence_and_uniqueness_inequalities} is trivially verified $\forall \lambda \in (0,1)$.
\end{example}

\begin{example}[Constant permanent impact with time-varying temporary impact]
\label{example:constantpermanentCT}

When one couples a constant permanent impact $\dot{\theta}_t=0$ with a temporary impact varying in time $\dot{\eta}_t\ne0$, the optimal execution strategy is no longer constant. The Euler-Lagrange equations are
\begin{align}
    \begin{cases}
    \eta_t \ddot \zeta_t+\dot \eta_t\dot \zeta_t&=0,\\
    \zeta_{t=0}&=Q,\\
    \zeta_{t=T}&=0.
\end{cases}
\end{align}
The first equation ensures that the derivative of $\eta_t\dot\zeta_t$ is zero. The solution is given by
\begin{eqnarray}
\label{CT_and_DT_are_not_consistent}
     \dot \zeta_{t}=-\frac{Q}{||\eta||_{-1}} \frac{1}{\eta_t},
\end{eqnarray}
with the $L^{-1}$ norm defined as $||\eta||_{-1}\equiv \int_0^T \eta_t^{-1}dt$. One trades faster when impact is small and slower when impact is large. The solution is unique, and it is derived using the boundary conditions and applying integration rules. To find this existence and uniqueness, we did not have to leverage Proposition \ref{proposition:ExistenceandUniquenessCT}. Notice that $ \dot \zeta_{t} < 0$ is negative on $[0,T]$, so the optimal solution avoids transaction-triggered price manipulation. Consequently, price manipulation is not possible in this setting.
\end{example}

\begin{example}[Constant temporary and linearly increasing permanent impact]

%In the previous example, we observed that a time-varying temporary impact combined with a constant permanent impact yields a unique solution characterized by monotonic inventory. 
Let us now look at the reverse scenario of the previous example with constant temporary impact $\eta > 0$ and permanent impact that varies over time. %Example \ref{secton3:CT}.\ref{example:linearlydecreasingpermanentCT} outlines straightforward sufficient conditions under which the Optimal Execution Problem either lacks a solution or exhibits execution strategies with sign changes.

Instead of solving the general case of time-varying permanent impact, we consider the linear case
\begin{equation}
\theta_t=a+b\cdot t,
\label{eq:linearlyvaryingimpact}
\end{equation}
where $b>0$ and $a>0$. This impact regime gives the % optimization problem:
%\begin{align}
%    \zeta^{\star}\in \arg \max_{\zeta\in \Xi} \int_0^T(\theta_{t}  \dot \zeta_{t}\zeta_t-\eta  \dot \zeta_{t}^2)dt
%\end{align}
%We apply 
Euler-Lagrange equations
\begin{align}
    \begin{cases}
        \ddot \zeta_{t} - \gamma\zeta_{t}&=0,\\
        \zeta_{t=0}&=Q,\\
    \zeta_{t=T}&=0.
    \end{cases}
\end{align}
where $\gamma =\frac{b}{2\eta}$ \footnote{The quantity \(\gamma\) is central in our analysis and will be studied and presented more generally in Proposition \ref{proposition:CT_is_well_defined}.}. The optimal solution is given by
\begin{eqnarray}
\zeta_t=Q \frac{\sinh [\sqrt{\gamma} (T-t)]}{\sinh[\sqrt{\gamma} T]}.
\end{eqnarray}

Notably, this solution aligns with that of a risk-averse investor facing constant permanent impact. This mapping is achieved by substituting $b \to \lambda \sigma^2$, where $\lambda$ represents the risk aversion coefficient of a CARA agent \cite{gueant2016financial}. Additionally, $ \dot \zeta_{t}<0$, which ensures the absence of price manipulation.
\label{example:tempconstantandincreasingpermanent}
\end{example}

\begin{example}[Constant temporary and linearly decreasing permanent impact]
\label{example:linearlydecreasingpermanentCT}

Equation \eqref{eq:linearlyvaryingimpact} also allows for decreasing permanent impact with $b<0$, but we must assume $a>|b|T$ to keep $\theta_t$ positive. The optimal solution takes the form
\begin{eqnarray}
    \zeta_t=Q \frac{\sin [\sqrt{|\gamma|} (T-t)]}{\sin[\sqrt{|\gamma|} T]}.
\end{eqnarray}
In this scenario, price manipulation occurs when the initial velocity is positive, $\dot \zeta_0>0$. This makes it optimal to buy at the beginning and sell (more than $Q$) later on. The onset of such behavior is observed when
\begin{eqnarray}
    |b|>\frac{\pi^2\eta}{2T^2}.
\end{eqnarray}
%Alternatively, the same condition can be expressed as
%\begin{eqnarray}
%    T>\sqrt{\frac{\eta\pi^2}{2|b|}}.
%\end{eqnarray}
It is worth noting that when $|b|=2n\pi^2\eta/T^2$ (with $n\in\mathbb{N})$, the Optimal Execution Problem lacks a solution. That particular setting of impacts violates Proposition \ref{proposition:ExistenceandUniquenessCT}. %However, this impact scenario is isolated, $\{(b,\eta)||b| = 2n\pi^2\eta/T^2, n^{2}\in\mathbb{N}\}$ are the unique points where the Optimal Execution Problem \eqref{problem:OTContinuousTime} does not have a solution.

For $|b| > \frac{\pi^2\eta}{2T^2}$, the optimal solution oscillates between positive and negative values of $\zeta_t$, indicating transaction-triggered price manipulation.

It is important to note that transaction-triggered price manipulation does not necessarily imply negative transaction costs. The cash in excess is
\begin{align}
     \frac{Q^2}{2} \left(-a + \frac{b \cot\sqrt{|\gamma|}}{\sqrt{|\gamma|}}\right)\,.
\end{align}
As expected, for constant permanent impact, this quantity is negative and reduces to the standard quadratic cost $-\frac{aQ^2}{2}$. However, there are cases where the cash in excess is positive. For example, with $T=1$, $a=5$, and $\eta=0.25$, the excess cash is positive when $b<-4.09$, that is,the cash obtained with the liquidation is greater than the initial mark-to-market value. Thus, in this simple case, we have a proper price manipulation.
\end{example}

\subsection{Second order conditions and well-posedness}\label{subsection:CT_second_order}

We now examine the well-posedness of the cash maximization problem, and in particular the conditions under which the extremal point of our calculus of variations problem is a maximum. Specifically, we provide both necessary and sufficient conditions for our Optimal Execution Problem \eqref{problem:OTContinuousTime} to qualify as a maximization problem. Additionally, we establish impact regimes in which the solution to the Euler-Lagrange equation represents the unique maximum for the problem.

%\bigbreak 

%We recall the optimal execution problem in continuous time 
%\begin{eqnarray}
%\zeta^{\star} = \arg \max_{\zeta\in\Xi}\int_0^T(\theta_t\dot \zeta_t\zeta_t-\eta_t\dot \zeta_t^2)dt
%\end{eqnarray}
%where $\Xi = \{\zeta\in\mathcal{C}^{2}(\mathbb{R},\mathbb{R}): \zeta_{t\rightarrow0}=Q\text{ and }\zeta_{t \rightarrow T}=0.\}$

\begin{proposition}
\label{proposition:loclaminimumCalculusofVariations}
    If an extremal point to$\mathcal{J}$ is a local maximum, then necessarily
    \begin{eqnarray}\label{inequality:localminimumCalculusofVariations_ineq}
        \int_{0}^{T}\phi(t)\dot{\phi}(t)\theta_{t}\,dt \le \int_{0}^{T}[\dot{\phi}(t)]^{2}\eta_{t}\,dt,
\quad \forall \phi \in \mathcal{C}^{\infty}_{c}((0,T))
    \end{eqnarray}
    where $\mathcal{C}^{\infty}_{c}((0,T))$ is the space of smooth test
    functions with compact support in $(0,T)$ (in particular, vanishing near
    $0$ and $T$).
\end{proposition}

Since this property is difficult to verify in practice, the following proposition provides a sufficient condition easy to test.

\begin{proposition}
\label{proposition:CT_is_well_defined}
    Let $\bar{\eta}=\min_{0 \le t \le T} \{ \eta_{t} \}$ and let
    $\bar{\theta}=\max_{0 \le t \le T}\max\{-\dot \theta_{t},0\}$. Define
    $\gamma=\frac{\bar{\theta}}{2\bar{\eta}}$. Assuming the Optimal Execution
    Problem has a unique solution, if $\sqrt{\gamma}T < \pi$, then it is a
    maximum, and the Optimal Execution Problem is weakly well-posed. 
\end{proposition}

The proof of such a result is based on classical perturbation method from calculus of variation, and the Sturm-Picone lemma, see \ref{lemma:SturmPicone} and proof \ref{proof:proposition:CT_is_well_defined}. 

%\begin{remark}
    Adding the assumptions of Proposition \ref{proposition:ExistenceandUniquenessCT} to Proposition \ref{proposition:CT_is_well_defined} we ensure that the Optimal Execution Problem is a maximization problem, with a unique solution, and this solution is the unique solution of the non-autonomous boundary value problem \eqref{eq:extremal_point_BVP}.
%\end{remark}

\medskip 
Proposition~\ref{proposition:CT_is_well_defined} establishes a second-order condition that not only ensures well-posedness of the optimal execution problem but also offers economic insight into its structural constraints under time-varying market impact. The inequality \(\sqrt{\gamma}T < \pi\) implies that the impact ratio \(\gamma = \frac{\bar{\theta}}{2\bar{\eta}}\) must be bounded, where \(\bar{\theta}\) denotes the maximal rate of decay of the permanent impact and \(\bar{\eta}\) the minimal level of temporary impact. 
This condition enforces a bound on how rapidly the permanent impact \(\theta_t\) may decline relative to the baseline execution cost. Economically, since temporary impact reflects the immediate cost of trading, it is always incurred. However, the permanent impact—capturing the longer-term informational footprint of trades—must not decay too quickly; otherwise, agents could exploit this dynamic to generate arbitrage-like strategies. Mathematically, this constraint prevents the emergence of change of concavity in the value function.

\begin{remark}
    Proposition \ref{proposition:loclaminimumCalculusofVariations} establishes a necessary condition for $\theta$ and $\eta$, when $\xi$ is a local maximum. In order to be rigorous, we have to prove that the constraint $\sqrt{\gamma}T<\pi$ does not contradict the inequality \eqref{inequality:localminimumCalculusofVariations_ineq}. 
    Note that using integration by parts and the Poincar\'e inequality with $d=T$ (diameter of the domain) and $p=2$, we get 
    \begin{align*}
        \int_{0}^{T}\phi(t)\dot \phi(t)\theta_{t}dt &= \frac{1}{2}\int_{0}^{T}\phi(t)^{2}(-\dot \theta_{t})dt \\
        &\le\frac{\bar{\theta}}{2}\int_{0}^{T}\phi(t)^{2}dt\\
        &\le\frac{\bar{\theta}T^{2}}{2\pi^{2}}\int_{0}^{T}(\dot\phi(t))^{2}dt\\
        &\le\frac{\gamma T^{2}}{\pi^{2}}\int_{0}^{T}(\dot\phi(t))^{2}\eta_{t}dt \\
        \int_{0}^{T}\phi(t)\dot \phi(t)\theta_{t}dt&<\int_{0}^{T}(\dot\phi(t))^{2}\eta_{t}dt.\\
    \end{align*}

Thus we proved that our assumption does not contradict the necessary conditions for $\zeta$ to be a local maximum, it is a sufficient condition for preserving Inequality \eqref{inequality:localminimumCalculusofVariations_ineq}. This finding has significant implications for the physical interpretation of our model. While Proposition \ref{proposition:loclaminimumCalculusofVariations} establishes a sufficient condition for the cash maximization problem, saturating Poincar\'e's inequality, turns our condition from being sufficient to both necessary and sufficient in certain regimes. Consequently, the inequality becomes optimal, and cannot be relaxed.
\end{remark}

\begin{remark}
\label{remark:comparison_sampling_admissibility}
Proposition~\ref{proposition:CT_is_well_defined} is a purely continuous-time
sufficient condition based on uniform bounds $\eta(t)\ge \bar\eta$ and
$-\dot\theta(t)\le \bar\theta$, together with the sharp constant $\pi$ in the
underlying Sturm--Picone/Poincar\'e argument.

By contrast, Proposition~\ref{prop:CT_from_Bmatrix} is \emph{sampling-induced}:
it arises from requiring \emph{strong discrete-time well-posedness on every refined sampling grid}.
It yields the pointwise constraint \eqref{eq:CT_weighted_slope}, which implies the global
second-order inequality \eqref{inequality:localminimumCalculusofVariations_ineq}.

Moreover, under the same uniform bounds, \eqref{eq:CT_weighted_slope} is implied by
\[
2\bar\eta \;\ge\; \sup_{t\in[0,T]}\bar\theta\int_0^t (T-s)\,ds
= \frac{\bar\theta T^{2}}{2},
\qquad\text{i.e.}\qquad \sqrt{\gamma}\,T \le \sqrt{2},
\]
which is stricter than $\sqrt{\gamma}\,T<\pi$.
This reflects the fact that sampling-robust strong well-posedness is more
demanding than weak well-posedness in continuous time.
\end{remark}

\medskip
The continuous-time formulation also provides a bridge back to our discrete-time results.
Proposition~\ref{proposition:_ct_limit_for_dt_model} makes precise the continuous-time
limit of the sampled discrete model (with the canonical $\eta\mapsto \eta/\Delta t$
scaling). In this sense, requiring the $B$-matrix property of the discrete impact
matrix on every refined grid (as in Proposition~\ref{prop:CT_from_Bmatrix}) enforces a
\emph{sampling-robust} notion of strong well-posedness: it rules out negative
transaction costs and prevents transaction-triggered price manipulation at any
execution frequency. This requirement is, however, conservative: it yields sufficient
conditions but not a full characterization of all impact settings that guarantee
strong (or even weak) well-posedness.

For decreasing permanent impact, our results lead to a simple hierarchy of sufficient
criteria. Under the uniform bounds of Proposition~\ref{proposition:CT_is_well_defined},
the pointwise admissibility constraint \eqref{eq:CT_weighted_slope} implies
\[
\gamma=\frac{\bar\theta}{2\bar\eta}\ \le\ \frac{2}{T^{2}},
\qquad\text{equivalently}\qquad
\sqrt{\gamma}\,T\ \le\ \sqrt{2},
\]
which in turn guarantees the global second-order inequality
\eqref{inequality:localminimumCalculusofVariations_ineq}. By contrast,
Proposition~\ref{proposition:CT_is_well_defined} only requires
\[
\gamma\ <\ \frac{\pi^{2}}{T^{2}},
\]
which is weaker by a factor $\pi^{2}/2$ and ensures weak well-posedness.
The two conditions are therefore complementary: the former is tailored to
sampling-robust strong well-posedness, whereas the latter is a purely continuous-time
criterion for weak well-posedness. Tightening the gap between these regimes and
deriving sharper criteria for strong well-posedness are natural directions for future work.

\subsection{Positivity, monotonicity, and price manipulation}
\label{subsection:Qualitative_insights}

We now investigate the qualitative aspects of the solution, in particular its positivity and monotonicity. We aim to identify necessary and/or sufficient conditions on $\theta$ and $\eta$ to prevent price manipulation and/or transaction-triggered price manipulation.

\begin{proposition}
\label{proposition:boundedbelow}
    Let $\bar{\eta}=\min_{0 \le t \le T} \{ \eta_{t} \}$, and let
    $\bar{\theta}=\max_{0 \le t \le T}\max\{-\dot \theta_{t},0\}$. Define
    $\gamma=\frac{\bar{\theta}}{2\bar{\eta}}$. Assuming the Optimal Execution
    Problem has a unique solution, if $\sqrt{\gamma}T < \pi$, then the optimal
    solution $\zeta$ is non-negative.
\end{proposition}

%\begin{remark}
    Notice that under the conditions of Proposition \ref{proposition:boundedbelow}, the inventory associated with the optimal strategy becomes zero only once at time $T$, or if there exists $t_{0}$ such that $\zeta_{t_{0}}=0$, then $\zeta_{t}=0$ for all $t\in[t_{0},T]$.
%\end{remark}

We now state several propositions regarding the properties of the solution in different cases of increasing or decreasing temporary and permanent impact. Table \ref{tab:impactAnalysis} summarizes the content of these propositions.

\begin{figure}
        \centering
        \includegraphics[scale=.6]{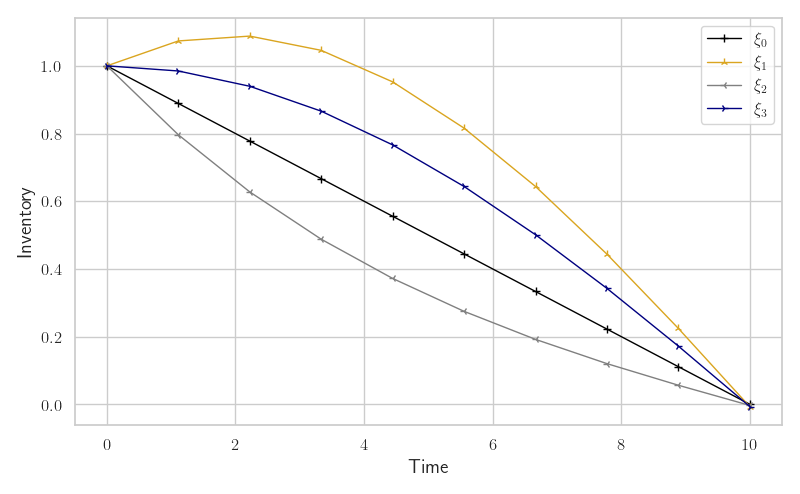}
        \caption{Optimal execution strategies. We consider only exponentially varying impacts such that $\theta_t = \beta_{\theta}e^{\alpha_{\theta}t}$ and $\eta_t = \beta_{\eta}e^{-\alpha_{\eta}t}$. $\pmb{\xi_0}$ is the TWAP strategy. $\pmb{\xi_1}$: is the optimal strategy with $\alpha_{\theta} = -1$, and $\alpha_{\eta} = -0.2$.  $\pmb{\xi_2}$: is the optimal strategy with $\alpha_{\theta} = 0.6$, and $\alpha_{\eta} = 0.2$. We have considered in every case $\frac{\beta_{\theta}}{\beta_{\eta}}=0.1$. $\pmb{\xi_3}$: is the optimal strategy with $\alpha_{\theta} = -0.5$,  and $\alpha_{\eta} = -0.2$.}
        \label{fig:possibletrajectories}
\end{figure}

\begin{proposition}
\label{proposition:IncreasingImpactsCT}
    If both permanent and temporary impact are increasing, the optimal solution is monotonic, bounded from above by the TWAP strategy and convex. Consequently, no transaction-triggered price manipulation or price manipulation is possible.
\end{proposition}

%\begin{remark}
    Proposition \ref{proposition:IncreasingImpactsCT} aligns with Example \ref{secton3:CT}.\ref{example:tempconstantandincreasingpermanent}, where we considered constant temporary and linearly increasing permanent impact. In that example, the optimal solution was a strictly decreasing, consistently staying below the TWAP strategy.
%\end{remark}

\begin{proposition}
\label{proposition:ConstantTemporaryDecreasingPermanentCT}
    If temporary impact is constant, permanent impact is decreasing, and $\sqrt{\gamma}T<\pi$, then the optimal solution is non-negative. It exhibits at most two phases of execution. The optimal solution can be either monotonic, thereby ruling out transaction-triggered price manipulation, or its velocity can change sign exactly once during execution. The optimal solution is always concave, and bounded from below by TWAP. 
\end{proposition}

%\begin{remark}
    Proposition \ref{proposition:ConstantTemporaryDecreasingPermanentCT} establishes a necessary and sufficient condition for a model with constant temporary impact and decreasing permanent impact. Specifically, if the initial velocity of a sell program is negative, then the optimal solution is monotonic, and no transaction-triggered price manipulation is possible. This condition, in turn, implies the absence of price manipulation. 
%\end{remark}

\begin{proposition}
\label{proposition:Temporary_decrease_Permanent_increase}
Assume permanent impact is strictly increasing and temporary impact is strictly
decreasing. Assume moreover that the Optimal Execution Problem has a unique
solution. Then the optimal solution is monotone (it involves only sell trades),
so there is no transaction-triggered price manipulation, and hence no price
manipulation.
\end{proposition}

\begin{proposition}
\label{proposition:DecreasingImpacts}
    Suppose that both permanent and temporary impact are decreasing. Given the assumptions of Proposition \ref{proposition:boundedbelow}, if a solution exists for the Optimal Execution Problem, it cannot be locally convex. Additionally, the inventory is consistently non-negative. Furthermore, in cases where the optimal solution allows for transaction-triggered price manipulation, the buy phase must be concave. The TWAP strategy bounds the optimal solution from below. 
\end{proposition}

%\begin{remark}
    It is worth noticing that Proposition \ref{proposition:DecreasingImpacts} states that the optimal strategy exhibits two possible patterns. On the one hand, it can be monotonically decreasing and concave. On the other hand, the trading profile can involve two execution phases: a first one buying with a concave strategy, then a second one selling in a concave manner. Throughout these scenarios, there is always a non-negative inventory.
%\end{remark}

\begin{proposition}
\label{proposition:Temporary_increase_Permanent_decrease}
    Let permanent impact be decreasing, and temporary impact increasing. Choose $\sqrt{\gamma}T<\pi$ and assume that there is a unique solution to Eq. \eqref{eq:extremal_point_BVP}. The Optimal Execution Problem has a unique solution, that is non-negative and can be either decreasing, or have one turning point.
\end{proposition}

%\begin{remark}
    Figure \ref{fig:possibletrajectories} illustrates the four possible types of optimal solutions, \emph{for a sell program with $Q=1$} and exponentially varying impact parameters, that our model can generate when imposing $\sqrt{\gamma}T < \pi$ from Proposition \ref{proposition:CT_is_well_defined}:
    \begin{enumerate}
        \item The black line refers to the TWAP strategy. This strategy has a constant velocity execution equal to $-\frac{1}{T}$. 
        \item The orange line represents a two regime execution strategy. The initial velocity is positive thus one starts the sale by buying, then one sells everything purchased plus the initial quantity to sell. This strategy has transaction-triggered price manipulation, and the inventory is concave, because $\ddot \zeta_{t} \le 0$ for all $t\in[0,T]$. 
        \item The grey line describes an optimal solution that has negative initial velocity. Note also that this solution is concave and thus $\ddot \zeta_{t} \le 0$ for all $t\in[0,T]$. This strategy is monotonic, but its concavity makes it always bounded from below by the TWAP strategy. 
        \item Finally the blue line represents a convex inventory strategy. The initial velocity is negative, and the strategy is always bounded from above by the TWAP strategy. Note that such a strategy is monotonic, and thus there is no transaction-triggered price manipulation. 
    \end{enumerate}
%\end{remark}

In conclusion, Table \ref{tab:impactAnalysis} displays the different patterns that can be observed depending on the monotonicity properties of permanent and temporary impact parameter.

%[fontsize=\small]
\begin{table}[ht]
\centering
\begin{tabular}{|l|c|c|c|}
\hline
\multirow{2}{*}{{\textbf{Temporary Impact}}} & \multicolumn{3}{c|}{\textbf{Permanent Impact}} \\ \cline{2-4} 
 & \textbf{Increasing} & \textbf{Decreasing} & \textbf{Constant} \\ \hline
\textbf{Increasing} &  cvx, mono &  tp &  cvx, mono \\ \hline
\textbf{Decreasing} &  mono &  ccv, tp & ccve, mono \\  \hline
\textbf{Constant} & cvx, mono &  ccv, tp &  const \\ \hline
\end{tabular}
\caption{Characteristics of the optimal solution in different impact settings. The abbreviations correspond to the following: cvx: convex, ccv: concave, ccve: convex or concave, mono: monotonic but non-linear, const: constant speed, tp: single turning point.}
\label{tab:impactAnalysis}
\end{table}

\section{Numerical examples}
\label{sec:numerical}
We now present some numerical analysis of the optimal execution and its properties when both temporary and permanent impact are time-varying. As discussed previously, the solutions of the Optimal Execution Problem exhibit diverse behavior contingent on the monotonicity of the impact. First, Proposition \ref{proposition:CT_is_well_defined} indicates that the problem is weakly well-posed when $\sqrt{\gamma}T < \pi$. This implies that, in cases where the permanent impact is increasing, the problem consistently takes the form of a maximization problem. On the contrary, if the permanent impact decreases too rapidly with respect to the minimal value of the temporary impact and the final time $T$, the well-posedness condition can be violated. While we have not explicitly demonstrated that $\sqrt{\gamma}T < \pi$ is a necessary condition for a strongly well-posed problem, we have observed that it is imperative in order to ensure the no-transaction-triggered price manipulation property in general. Note that it is sufficient in scenarios involving linear decreasing permanent impact, see Example \ref{secton3:CT}.\ref{example:linearlydecreasingpermanentCT}. 

For a given dynamics of permanent and temporary impact, assuming that the Optimal Execution Problem is well-posed, solving it consists in solving a second order non-autonomous boundary value problem with Dirichlet conditions, see Eq. \eqref{eq:extremal_point_BVP}. While such a problem may not be solvable analytically, numerical methods can be employed. Here we present the results for exponentially decaying impact parameters, while in \ref{app:powerlaw} give similar results for power-law decaying parameters. We employ the shooting method \cite{osborne1969shooting} to solve the boundary value problem. This approach transforms the BVP into a first-order ordinary differential equation (ODE) by guessing an initial trading velocity and iteratively refining it through dichotomy.

%\subsection{Exponentially decreasing impacts}

Consider the decreasing monotonic impact functions
\begin{align}
    \theta_{t} = \beta_{p}\exp\left(-\alpha_{p}\frac{t}{T}\right), && \eta_{t} = \beta_{tp}\exp\left(-\alpha_{tp}\frac{t}{T}\right).
\end{align}
Since $-\dot \theta_{t} \le \alpha_{p}\beta_{p}=\bar \theta$ and $\bar \eta=\beta_{tp}\exp(-\alpha_{tp})$, by Proposition \ref{proposition:CT_is_well_defined} the Optimal Execution Problem is a maximization problem when 
\begin{eqnarray}
    \sqrt{\frac{\alpha_{p}\beta_{p}}{2T\beta_{tp}\exp(-\alpha_{tp})}}T < \pi.
\end{eqnarray}
As before, setting $\beta_{tp}=T\kappa\beta_{p}$ and $T=1$, one can rewrite the inequality as
\begin{eqnarray}
\label{inequality:expo_exponents}
    \alpha_{p} < 2\kappa\pi^{2}\exp(-\alpha_{tp}).
\end{eqnarray}

We consider two cases depending on the equality between $\alpha_{p}$ and  $\alpha_{tp}$. 
\begin{example}[The case of $\alpha_{p}\neq \alpha_{tp}$]
Using the first-order conditions and the boundary conditions, without loss of generality we set $T=1$, and get
\begin{align}
\begin{cases}
    \ddot \zeta_{t} &= \alpha_{tp}\dot \zeta_{t} - \frac{\alpha_{p}}{2\kappa}\exp((\alpha_{tp}-\alpha_{p})t)\zeta_{t}, \\
        \zeta_{0}&=Q,\\
        \zeta_{T}&=0.
    \end{cases}
\end{align}
This boundary value problem has a solution when there exists $\lambda \in (0,1)$ such that 
\begin{align}
\label{conditions:expo_different_exponents}
        \begin{cases}
            \frac{\alpha_{p}}{2\kappa(\alpha_{tp}-\alpha_{p})}\left[1-\exp(\alpha_{tp}-\alpha_{p})\right] &\le \log(1+2\lambda), \\ 
            \alpha_{tp}\le \log(1+2\lambda).
        \end{cases}
\end{align}
\begin{figure}
        \centering
        \includegraphics[scale=0.37]{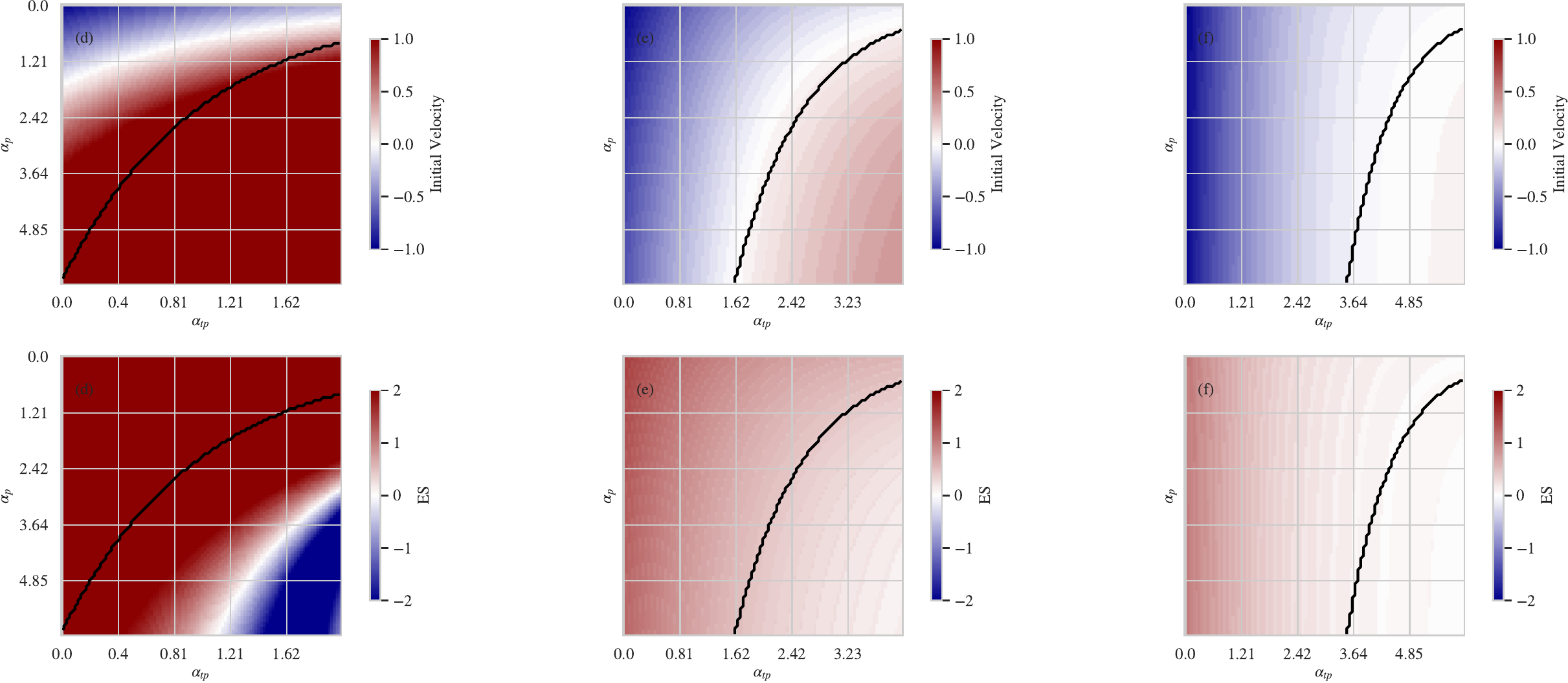}
        \caption{ \textbf{First row:} Comparison of the initial velocities with respect to different market conditions, and admissible regimes. \textbf{Second row:} Expected implementation shortfall associated to the optimal strategies. \textbf{Remark:} Each column is associated to a parameter $\kappa$ (first column with $\kappa=0.1$, second column with $\kappa=1$, third column with $\kappa=10$). The region on the left of the black line is the admissible region defined by Proposition \ref{proposition:CT_is_well_defined}.}
        \label{fig:init_velocities_diff_expo}
\end{figure}

It is not possible to characterize analytically all impact settings that allow transaction-triggered price manipulation and negative transaction costs. We propose instead to explore the parameter space by examining the initial trading velocities associated with optimal execution strategies in various impact scenarios. Specifically, we investigate the cost values as indicators of the impact settings.

The first row of Figure \ref{fig:init_velocities_diff_expo} presents the initial trading velocities as a function of $\alpha_{tp}$ and $\alpha_{p}$, while the second row shows the costs. The first column considers $\kappa=0.1$. In the admissible region, defined by Eq. \eqref{inequality:expo_exponents} and corresponding to the left part delimited by the black line, we observe initial trading velocities greater than $-1$, in agreement with Proposition \ref{proposition:DecreasingImpacts}. The admissible region allows for monotonic concave solutions, represented by the blue region, and non-monotonic concave solutions, associated with transaction-triggered manipulation, represented by the red region. Numerically we find that the transition between the two regions is smooth and is described by a power law relation between $\alpha_{p}$ and $\alpha_{tp}$. %that the maximum value of $\alpha_{p}$ for being in the blue region decreases as a power law with $\alpha_{tp}$ and the transition is . A smooth transition from monotonic concave solutions (blue region) to non-monotonic concave solutions (red region) is noticeable. As $\alpha_{tp}$ increases, the admissible values of $\alpha_{p}$ decrease as exponentially fast, consistent inequality \ref{inequality:expo_exponents}. 
In the admissible region, the costs are positive, indicating positive expected implementation shortfall. This demonstrates that the exponentially decreasing impact model with different exponents does not permit manipulation and is strongly well-posed in the admissible region.

The second and third columns consider $\kappa=1$ and $\kappa=10$, respectively. Focusing on the admissible regions (left side of the black line), we observe that regions allowing transaction-triggered price manipulation (in red) are less prevalent. Interestingly, the red region shifts to larger values of $\alpha_{tp}$ and extends beyond the admissible region. Additionally, the maximal value of $\alpha_{p}$ for being in the blue region decreases exponentially fast with $\alpha_{tp}$, aligning with the definition of $\gamma$ and Proposition \ref{proposition:CT_is_well_defined}. As $\kappa$ increases, the solution is mostly monotonic and concave in the admissible region, thus transaction-triggered price manipulations become rarer.  %The range of concavities for monotonic solutions widens as we increase $\kappa$, consistently with the growing prominence of blue regions in the admissible region with increasing $\kappa$.

Notably, negative transaction costs do not exist for the admissible strategies, as observed for $\kappa=0.1$. This is expected, given that with relatively large values of $\kappa$, the admissible region only allows monotonic solutions, indicating no transaction-triggered price manipulation and, consequently, positive expected implementation shortfall by Proposition \ref{proposition:NoPTPimpliesNOPM} \cite{alfonsi2010optimal}.

In conclusion, the time-varying Almgren-Chriss impact model in continuous time, characterized by exponentially decreasing impacts with different speeds can indeed lead to transaction-triggered price manipulation. However, it does not admit negative expected implementation shortfall. Moreover, as the value of $\kappa$ increases, indicating $\frac{\eta_{0}}{T}\gg\theta_{0}$, the model also eliminates any potential form of transaction-triggered price manipulation.%, ensuring the absence of negative transaction costs.
\end{example}

\begin{example}[The case of $\alpha_{tp}=\alpha_{p}=:\alpha$]
We obtain a linear autonomous and homogeneous second order differential equation, and the solution is given by
\begin{align}
    \zeta_{t} = \frac{Q e^{\sqrt{\alpha^2 -2 \frac{\alpha}{\kappa}}}}{e^{\sqrt{\alpha^2 -2 \frac{\alpha}{\kappa}}}-1} e^{\frac{t}{2} \left(\alpha-\sqrt{\alpha^2-2\frac{\alpha}{\kappa}}\right)}-\frac{Q}{e^{\sqrt{\alpha^2 -2 \frac{\alpha}{\kappa}}}-1} e^{\frac{t}{2} \left(\sqrt{\alpha^2-2\frac{\alpha}{\kappa}}+\alpha\right)}.
\end{align}
A prerequisite to determine the optimal execution strategy is \( \ \alpha \neq \frac{2}{\kappa} \). Otherwise the boundary conditions are satisfied if and only if $Qe^{\frac{\alpha}{2}}=0$, which is impossible since $\alpha \ge 0$ and $Q>0$. Hence the necessary and sufficient conditions for having a well-posed Optimal Execution Problem are  
\begin{align}
\label{inequality:same_expo_exponents}
    \begin{cases}
        \alpha e^{\alpha} < 2\kappa \pi^{2}, \\
        \alpha \neq \frac{2}{\kappa}.
    \end{cases}
\end{align}

In this context we have a closed form solution, and  we can completely characterize the possible manipulation regimes. To explore those regimes, Figure \ref{fig:init_velocities_equal_expo_expo} shows a heatmap with \(\alpha\) and $\kappa$ on the axes. This emphasizes the regions of the phase space where transaction-triggered price manipulation and price manipulation occur.

In the left panel of Figure \ref{fig:init_velocities_equal_expo_expo}, the color bar corresponds to the initial trading velocities for optimal execution strategies. The region to the left of the black line identifies the admissible market regimes satisfying the inequality (\ref{inequality:same_expo_exponents}). For a fixed $\kappa$, $\alpha$ is bounded, and the maximum value of $\alpha$ increases as a power law.

The admissible region defined by Proposition \ref{proposition:CT_is_well_defined} encompasses solutions, which are small perturbations of the TWAP (light blue regions where the slope is close to $-1$) and monotonic and more concave strategies (medium blue regions). In this admissible region, the associated impact settings produce monotonic optimal execution strategies, ensuring no negative expected implementation shortfall by Proposition \ref{proposition:NoPTPimpliesNOPM}. This property is evident in the right panel of Figure \ref{fig:init_velocities_equal_expo_expo}, where the cost is positive.%; the cash generated by the strategy is less than the initial inventory value at time $0$.

Looking again at the left panel of Figure \ref{fig:init_velocities_equal_expo_expo}, we observe a red area in the admissible region corresponding to concave strategies with positive trading velocity. As anticipated by Proposition \ref{proposition:DecreasingImpacts}, the solution is bounded from below by the TWAP. The red admissible region becomes smaller when $\kappa$ increases, consistently with what was observed in the previous setting. These strategies exhibit a unique local maximum, involving buying in the initial phase of execution and selling thereafter.

Despite this presence of transaction-triggered price manipulation, the right panel of Figure \ref{fig:init_velocities_equal_expo_expo} indicates that optimal strategies associated with parameters $\alpha$ and $\kappa$ anywhere within the admissible region generate positive transaction costs, and therefore show no arbitrage.

\begin{figure}
        \centering
        \includegraphics[scale=0.45]{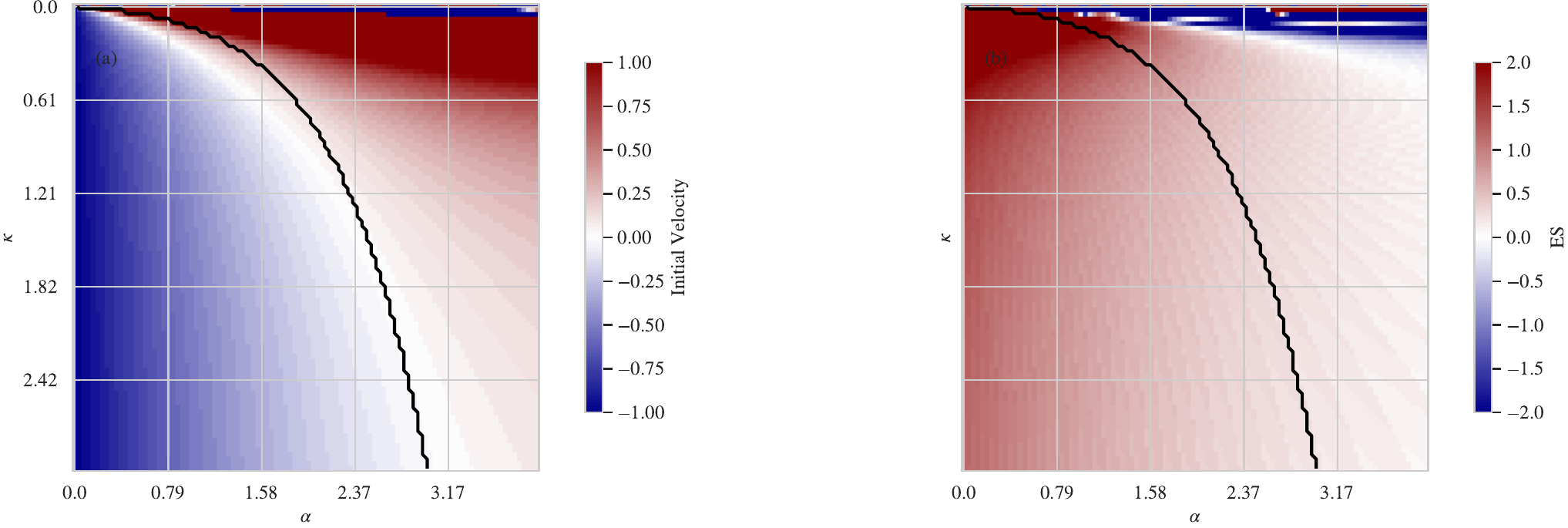}
        \caption{\textbf{Left:} Comparison of the initial velocities with respect to different market conditions, and admissible regimes. \textbf{Right:} Cash discrepancy associated to the optimal strategies. \textit{Both figures:} The region to the left of the black line is the admissible region defined by Proposition \ref{proposition:CT_is_well_defined}.}
        \label{fig:init_velocities_equal_expo_expo}
\end{figure}

\end{example}

\section{Conclusion}\label{sec:conclusions}

This paper addressed the Optimal Execution Problem defined by the time-varying Almgren-Chriss model. Assuming the impact dynamics are deterministically time-varying, the solutions of the Optimal Execution Problem are static, i.e., known at time zero by the trader. Despite this simplification, the problem remains very rich and difficult.

Leveraging both discrete and continuous time approaches, we found two related sufficient conditions for making the problem well-posed: a fairly restrictive condition ensuring that the problem is strongly well-posed and a similar, but more relaxed one for being weakly well-posed. Both conditions limit the speed of decrease of the permanent impact with respect to the level of temporary impact in the model. The insight here is that if one can apply an oversized initial positive permanent impact by buying at the beginning of a sell program, the cost of the subsequent sells may not completely offset this initial gain. In more extreme cases, the algorithm might even appear to achieve negative costs (thus a profit) purely by executing. Both of these behaviors are highly undesirable in live trading, and our paper provides sufficient conditions to preclude them. These conditions can serve as prototypes for more complex models in practical applications.

We also studied the qualitative behavior of the optimal solutions using necessary conditions given by the Euler-Lagrange equation and impact monotonicity. We have ruled out many impossible shapes of solutions. Optimal solutions of a sell program can be either convex monotonic, linear (TWAP), concave monotonic, or concave with one turning point (first buys then sells). No admissible impact setting gives negative solutions, whereby liquidating a long position would pass through a net short position. We have equally ruled out solutions oscillating between multiple phases of buys and sells. When executing a well-defined net order, such trading could be viewed as market manipulation by regulators. Therefore, sensible behavior must be guaranteed in practice, and it is helpful to have insight into why Almgren-Chriss is well-behaved in this regard.

We supplemented our results with a numerical study when impacts are exponentially decreasing. We chose this rather extreme scenario to clearly demonstrate how to explore the parameter space in order to provide insights into the model's robustness against both types of manipulation. Our findings are coherent with the theoretical results, and we are able to completely discriminate between pure and transaction-triggered price manipulation strategies. 

Future work should extend our study to cross-impact models, see \cite{mastromatteo2017trading, schneider2019cross}. In addition, our theoretical study should be tested on real data in order to understand whether, depending on the market conditions, the market can temporarily stay in the critical region, with transaction-triggered price manipulation, and over time revert back to regions associated to strongly well-posed Optimal Execution Problems. 

\section*{Acknowledgments}
FL acknowledges support from the grant PRIN2022 DD N. 104 of February 2, 2022 ”Liquidity and systemic risks in centralized and decentralized markets”, codice proposta 20227TCX5W - CUP J53D23004130006 funded by the European Union NextGenerationEU through the Piano Nazionale di Ripresa e Resilienza (PNRR).

\bibliography{references_mor}

\BeginAppendix
    
\section{Proofs}\label{app:proofs}

This appendix section presents all the proofs of propositions and lemmas utilized in the paper.

\subsection{Discrete time market impact model}

In this subsection, we provide the proofs related to the discrete impact model and the resulting Optimal Execution Problem.

\subsubsection{Proof of Proposition \ref{proposition:CostFunctionDT}}
We have
\[
\mathcal{C}(\xi)
= \mathbb{E}_0\left[ Q S_0 + \sum_{i=1}^N \xi_i \tilde S_i \right]
= Q S_0 + \mathbb{E}_0\left[\sum_{i=1}^N \xi_i \tilde S_i \right],
\]
so minimizing $\mathcal{C}$ is equivalent to minimizing
$\mathbb{E}_0\left[\sum_{i=1}^N \xi_i \tilde S_i\right]$.
\begin{align}
    \mathbb{E}_{0}\left[\sum_{i=1}^{N}\xi_{{i}}\tilde{S}_{{i}}\right] &= \mathbb{E}_{0}\left[\sum_{i=1}^{N}\xi_{{i}}(S_{i-1}+\eta_{{i}}\xi_{{i}})\right]\\
    &= \mathbb{E}_{0}\left[\sum_{i=1}^{N}\xi^{2}_{{i}}\eta_{{i}}\right]+\mathbb{E}_{0}\left[\sum_{i=1}^{N}\xi_{{i}}\left(\sum_{k=1}^{i-1}(S_{k}-S_{{k-1}})+S_{0}\right)\right]\\
    &=\mathbb{E}_{0}\left[\sum_{i=1}^{N}\xi^{2}_{{i}}\eta_{{i}}\right]+\mathbb{E}_{0}\left[\sum_{i=1}^{N}\xi_{{i}}\sum_{k=1}^{i-1}\theta_{k}\xi_{{k}}\right]+QS_{0}
\end{align}
where we used $S_k-S_{k-1}=\theta_k\xi_k+\varepsilon_k$ and the fact that the
noise terms have zero conditional expectation. Indeed,
\[
\sum_{i=1}^{N}\xi_{i}\sum_{k=1}^{i-1}\varepsilon_{k}
=\sum_{k=1}^{N-1}\varepsilon_{k}\sum_{i=k+1}^{N}\xi_{i}
=-\sum_{k=1}^{N-1}\zeta_{k}\varepsilon_{k},
\]
and $\zeta_k$ is $\mathcal{F}_{k-1}$-measurable while $\varepsilon_k$ is
independent of $\mathcal{F}_{k-1}$ with $\mathbb{E}[\varepsilon_k]=0$.
where $QS_{0}$, is $\mathcal{F}_{0}$ measurable, and is a constant. Moreover we can rewrite this quantity as follows 
\begin{eqnarray}
\mathbb{E}_{0}\left[\sum_{i=1}^{N}\xi_{{i}}\tilde{S}_{{i}}\right] = \frac{1}{2}\mathbb{E}_{0}\left[\xi^{T}A\xi\right]+QS_{0}
\end{eqnarray}
where $A$ is $\mathcal{F}_{0}$-measurable. Applying Jensen’s inequality to the convex function $x\mapsto x^{T}Ax$ the Optimal Execution Problem is equivalent to 

\begin{eqnarray}
    \xi \in \underset{\xi \in \Xi_{N,0}(Q)}{\arg \min} \left\{ \frac{1}{2}\xi^{T}A\xi\right\}
\end{eqnarray}
or equivalently
\begin{align}
    \label{problem:AC_ConstantImpacts}
    \xi \in \underset{\xi \in \Xi_{N,0}(Q)}{\arg \min} \left\{ \sum_{i=1}^{N}\eta_{{i}} \xi_{i}^{2}+\sum_{i=2}^{N}\xi_{i}\sum_{k=1}^{i-1}\theta_{{k}}\xi_{k}\right\}
    \end{align}

\subsection{Proof of Proposition \ref{proposition:DTProblemOptimalsolution}}

When $A$ is SPD the Optimal Execution Problem \ref{problem:DT} consists of minimizing a strictly convex function on a closed and convex hyperplane. Hence we obtain a strictly convex minimization problem with a unique solution. 

The Lagrangian of the problem is defined as follows
\begin{eqnarray}
\mathcal{L}(\xi,\lambda)= \frac{1}{2}\xi^{T}A\xi + \lambda(\pmb{1}^{T}\xi+Q)
\end{eqnarray}
By Lagrange theorem \cite{boyd_vandenberghe_2004} for convex functions, this solution is the unique stationary point that satisfies the Lagrange optimality conditions:
\begin{align}
    \begin{cases}
    \pmb{1}^{T}\xi^{\star}&=-Q \\
    \lambda &= \frac{Q}{\pmb{1}^{T}A^{-1}\pmb{1}}\\
    \xi^{\star} &= -\,\frac{Q}{\pmb{1}^{T}A^{-1}\pmb{1}} A^{-1}\pmb{1}
    \end{cases}
\end{align}

\subsubsection{Proof of Proposition \ref{proposition:diagdom}}

$A$ is diagonally dominant if and only if \begin{eqnarray}
    2\eta_{i} > \sum_{j=1}^{i-1}\theta_{j}+(N-i)\theta_{i}
\end{eqnarray}
If $A$ is diagonally dominant and symmetric, then it is SPD. Therefore by Proposition \ref{proposition:DTProblemOptimalsolution} the Optimal Execution Problem has a unique solution.

By the spectral theorem there exists $U\in\mathbb{O}_{N}$ such that $P=U^{T}\sqrt{D}$ and $PP^{T}=A$, where $D=diag(\{\lambda_{i}(A)\}_{i=1}^{N})$ with $\lambda_{i}(A)>0$ for all $i\in\{1,\dots,N\}$. This implies that the Optimal Execution Problem associated to $A$ does not admit price manipulation. 

\subsubsection{Proof of Proposition \ref{proposition:BmatrixNoTTPM}}

We assume that $A$ is a B-matrix. A symmetric B-matrix is symmetric positive definite \ref{lemma:BmatrixSPD}, thus invertible. By Proposition \ref{proposition:DTProblemOptimalsolution} Problem \ref{problem:DT} has a unique solution
\begin{eqnarray}
    \xi^{\star} &= -\frac{Q}{\pmb{1}^{T}A^{-1}\pmb{1}} A^{-1}\pmb{1}.
\end{eqnarray}
By lemma \ref{lemma:christensen}, if $A$ is a B-matrix then $A^{-1}\pmb{1}$ is non-negative element-wise. Hence $\xi$ has non-negative components, and the Optimal Execution Problem has a unique solution and does not admit transaction-triggered price manipulation.

\subsubsection{Proof of Proposition \ref{proposition:_ct_limit_for_dt_model}}

Recall that the impact matrix $A=(a_{ij})_{1\le i,j\le N}$ is given by
\eqref{eq:encoding_matrix}, i.e.
\[
a_{ii} = 2\eta_i,\qquad
a_{ij} =
\begin{cases}
\theta_j,& j<i,\\
\theta_i,& j>i,
\end{cases}
\qquad 1\le i\ne j\le N.
\]
Fix a row index $i\in\{1,\dots,N\}$. Its row sum is
\[
R_i = \sum_{j=1}^N a_{ij}
= 2\eta_i + \sum_{j=1}^{i-1}\theta_j + (N-i)\theta_i,
\]
so the corresponding row average is $m_i = R_i/N$.
Because the sequence $(\theta_k)_{k=1}^N$ is non-increasing and strictly
positive, we have
\[
\theta_1 \ge \theta_2 \ge \dots \ge \theta_N>0.
\]
In particular, for any fixed row $i$, all off-diagonal entries satisfy
$a_{ij}\le \theta_1$ for $j\neq i$, so the largest off-diagonal element
in row $i$ is
\[
\max_{j\ne i} a_{ij} = \theta_1.
\]
We first obtain a lower bound on $R_i$ which is uniform in $i$.
Since $(\theta_k)$ is non-increasing, for every $i$ we have
\[
\sum_{j=1}^{i-1}\theta_j
\;\ge\; (i-1)\theta_i
\;\ge\; (i-1)\theta_N,
\]
and also $\theta_i\ge\theta_N$. Thus
\begin{align*}
R_i
&= 2\eta_i + \sum_{j=1}^{i-1}\theta_j + (N-i)\theta_i \\
&\ge 2\eta_i + (i-1)\theta_i + (N-i)\theta_i \\
&= 2\eta_i + (N-1)\theta_i \\
&\ge 2\underline\eta + (N-1)\theta_N,
\end{align*}
where $\underline\eta = \min_{1\le k\le N} \eta_k$.
Dividing by $N$ gives
\[
m_i = \frac{R_i}{N}
\;\ge\;
\frac{2\underline\eta + (N-1)\theta_N}{N}
\qquad\text{for all } i=1,\dots,N.
\]
By assumption \eqref{inequality:BmatRestrictive},
\[
2\underline\eta \;\ge\; N\theta_1 - (N-1)\theta_N
\quad\Longleftrightarrow\quad
2\underline\eta + (N-1)\theta_N \;\ge\; N\theta_1.
\]
Hence
\[
m_i = \frac{R_i}{N}
\;\ge\;
\frac{2\underline\eta + (N-1)\theta_N}{N}
\;\ge\;
\theta_1.
\]
Because the off-diagonal entries in row $i$ are bounded above by
$\theta_1$, this implies
\[
m_i \;\ge\; \theta_1 = \max_{j\neq i} a_{ij}.
\]
If we strengthen \eqref{inequality:BmatRestrictive} to a strict inequality,
we obtain $m_i>\theta_1$ for all $i$, so condition (ii) in the definition
of a $B$-matrix is satisfied. Condition (i) also holds since
$R_i = N m_i > 0$. Therefore $A$ is a $B$-matrix.

\subsubsection{Proof of Proposition~\ref{prop:CT_from_Bmatrix}}

Fix $N\in\mathbb{N}$ and abbreviate $\theta_i=\theta_i^{(N)}$,
$\eta_i=\eta_i^{(N)}$, and $A=A^{(N)}$.
Since $\theta$ is strictly decreasing, the sampled sequence is strictly
decreasing:
$\theta_1>\theta_2>\cdots>\theta_N$.

For a fixed row $i$, the off-diagonal entries are
$\{a_{ij}:j\neq i\}=\{\theta_1,\dots,\theta_{i-1},\theta_i,\dots,\theta_i\}$,
hence
\[
\max_{j\neq i} a_{ij}=\theta_1.
\]
Because $A$ is a $B$-matrix, its row mean strictly dominates every off-diagonal
entry. Therefore, for each $i=1,\dots,N$,
\begin{equation}
\label{eq:Bmean_dom_theta1}
\frac{1}{N}\sum_{k=1}^{N}a_{ik}>\theta_1.
\end{equation}
A direct computation of the row sum gives
\[
\sum_{k=1}^{N}a_{ik}
=
2\eta_i+\sum_{k=1}^{i-1}\theta_k+(N-i)\theta_i.
\]
Insert this into \eqref{eq:Bmean_dom_theta1}, multiply by $T$, and use $\Delta t=T/N$:
\begin{equation}
\label{eq:discrete_Tscaled}
\frac{2T}{N}\eta_i+\Delta t\sum_{k=1}^{i-1}\theta_k+(T-t_i)\theta_i
> T\theta_1.
\end{equation}
By the scaling $\eta_i=\eta(t_{i-1})/\Delta t$, we have
$\frac{2T}{N}\eta_i=2\eta(t_{i-1})$. Also $\theta_1=\theta(0)$, and
$\theta_i=\theta(t_{i-1})$. Thus \eqref{eq:discrete_Tscaled} becomes
\begin{equation}
\label{eq:discrete_timegrid}
2\eta(t_{i-1})
+\Delta t\sum_{k=1}^{i-1}\theta(t_{k-1})
+(T-t_i)\theta(t_{i-1})
> T\theta(0).
\end{equation}

Now fix $t\in[0,T]$ and choose $i=i(N)$ such that $t_{i(N)-1}\to t$ and
$t_{i(N)}\to t$ as $N\to\infty$ (e.g.\ $i(N)=\lfloor Nt/T\rfloor+1$).
Since $\theta$ is continuous, the Riemann sums converge:
\[
\Delta t\sum_{k=1}^{i(N)-1}\theta(t_{k-1})\longrightarrow \int_{0}^{t}\theta(u)\,du.
\]
By continuity of $\eta$ and $\theta$,
\[
\eta(t_{i(N)-1})\to\eta(t),\qquad (T-t_{i(N)})\theta(t_{i(N)-1})\to (T-t)\theta(t).
\]
	Although \eqref{eq:discrete_timegrid} is strict for each $N$, the strictness
	can be lost in the limit. Taking $\liminf$ on both sides of
	\eqref{eq:discrete_timegrid} yields
\[
2\eta(t)+\int_{0}^{t}\theta(u)\,du+(T-t)\theta(t)\ge T\theta(0),
\]
which is \eqref{eq:CT_pointwise_admissibility}.

If $\theta\in C^{1}([0,T])$, then to obtain \eqref{eq:CT_weighted_slope}, note
that for $t\in[0,T]$,
\[
T\theta(0)-\int_{0}^{t}\theta(u)\,du-(T-t)\theta(t)
=\int_{0}^{t}(T-s)\,(-\dot\theta(s))\,ds,
\]
which follows by writing $\theta(0)-\theta(s)=\int_{0}^{s}(-\dot\theta(r))\,dr$
and applying Fubini's theorem. Evaluating \eqref{eq:CT_weighted_slope} at $t=T$
gives \eqref{eq:CT_terminal_condition}.

\subsubsection{Details for Remark~\ref{remark:comparison_sampling_admissibility}}

Let $\phi\in \mathcal{C}^{\infty}_{c}((0,T))$. Since $\phi(T)=0$,
$\phi(t)=-\int_t^T \dot\phi(u)\,du$, hence by Cauchy--Schwarz,
$\phi(t)^2\le (T-t)\int_t^T \dot\phi(u)^2\,du$.
Multiplying by $(-\dot\theta(t))/2$ and integrating over $t\in[0,T]$ gives
\[
\frac12\int_0^T (-\dot\theta(t))\,\phi(t)^2\,dt
\le
\frac12\int_0^T \dot\phi(u)^2
\left(\int_0^u (T-t)(-\dot\theta(t))\,dt\right)\,du
\le
\int_0^T \eta(u)\dot\phi(u)^2\,du,
\]
where the last inequality uses \eqref{eq:CT_weighted_slope}.
Finally, integration by parts yields
$\int_0^T \theta\,\phi\dot\phi
= \frac12\int_0^T \theta(\phi^2)'\,dt
= -\frac12\int_0^T \dot\theta\,\phi^2\,dt$
(since $\phi$ vanishes near $0$ and $T$), which concludes.

\subsection{Continuous time impact model}

In this subsection, we provide the proofs related to the continuous time impact model and the inherited Optimal Execution Problem.

\subsubsection{Proof of Proposition \ref{proposition:EMdiscretization}}

Fix a horizon $T>0$ and, for each $N\in\mathbb{N}$, define the uniform grid
\[
0 = t_0 < t_1 < \dots < t_N = T,
\qquad
\Delta t = t_k - t_{k-1} = T/N.
\]
\medskip\noindent
\textit{Discrete-time Almgren--Chriss model.}
For each $N$, consider a discrete strategy given by trades
$\xi^{(N)}_1,\dots,\xi^{(N)}_N$, and define the discrete inventory by
\[
\zeta^{(N)}_0 = Q, 
\qquad 
\zeta^{(N)}_k = \zeta^{(N)}_{k-1} + \xi^{(N)}_k,\quad k=1,\dots,N.
\]
The discrete-time impact model is
\begin{align}
S^{(N)}_k 
&= S^{(N)}_{k-1} + \theta^{(N)}_k\,\xi^{(N)}_k + \varepsilon^{(N)}_k,
\label{eq:disc-price}\\
\tilde S^{(N)}_k 
&= S^{(N)}_{k-1} + \eta^{(N)}_k\,\xi^{(N)}_k,
\label{eq:disc-exec}
\end{align}
where:
\begin{itemize}
  \item $\theta^{(N)}_k$ is the discrete \emph{permanent} impact coefficient,
  \item $\eta^{(N)}_k$ is the discrete \emph{temporary} impact coefficient,
  \item $\varepsilon^{(N)}_k$ is the discrete noise term.
\end{itemize}
The cash process in discrete time is updated as
\[
X^{(N)}_k 
= X^{(N)}_{k-1} - \xi^{(N)}_k\,\tilde S^{(N)}_k
= X^{(N)}_{k-1} - \xi^{(N)}_k\Big(S^{(N)}_{k-1} + \eta^{(N)}_k\xi^{(N)}_k\Big).
\]
\medskip\noindent
\textit{Step 1: Continuous-time interpolation of strategies.}
Define the continuous-time interpolation of the inventory:
\[
\zeta^{(N)}(t)
= \zeta^{(N)}_{k-1} 
   + \frac{t - t_{k-1}}{\Delta t}\,\xi^{(N)}_k,
\qquad t\in[t_{k-1},t_k),
\]
so that $\zeta^{(N)}(t_k)=\zeta^{(N)}_k$. On each interval, the trading rate is
piecewise constant:
\[
\dot\zeta^{(N)}_t
= \frac{d}{dt}\zeta^{(N)}(t)
= \frac{\xi^{(N)}_k}{\Delta t},
\qquad t\in(t_{k-1},t_k).
\]
\medskip\noindent
\textit{Step 2: Continuous-time impact functions induced by the discrete model.}
We now build piecewise-constant impact functions from the discrete
coefficients:
\[
\theta^{(N)}(t) = \theta^{(N)}_k,
\qquad
\eta^{(N)}(t) = \Delta t\,\eta^{(N)}_k,
\qquad
t\in[t_{k-1},t_k).
\]
The scaling in the second definition is crucial:
\[
\eta^{(N)}_k = \frac{\eta^{(N)}(t)}{\Delta t}
\qquad
\text{on }[t_{k-1},t_k).
\]
Intuitively: in discrete time, $\eta^{(N)}_k$ multiplies the \emph{block
size} $\xi^{(N)}_k$, while in continuous time $\eta_t$ multiplies the
\emph{trading rate} $\dot\zeta_t$.
Assume that, as $N\to\infty$, the stepwise functions converge uniformly to
continuous limits:
\[
\theta^{(N)} \to \theta, 
\qquad
\eta^{(N)} \to \eta
\quad\text{uniformly on }[0,T],
\]
where $\theta,\eta$ are continuously differentiable and strictly positive.
\medskip\noindent
\textit{Step 3: Continuous-time limit of the price process.}
From \eqref{eq:disc-price}, we can write
\[
S^{(N)}_k - S^{(N)}_{k-1}
= \theta^{(N)}(t_{k-1})\,\xi^{(N)}_k + \varepsilon^{(N)}_k.
\]
Introduce the piecewise-constant interpolation
\[
S^{(N)}(t) = S^{(N)}_k, \qquad t\in[t_k,t_{k+1}).
\]
Assume the noise is of the diffusive form
\[
\varepsilon^{(N)}_k = \sigma \sqrt{\Delta t}\,Z_k,
\]
where $(Z_k)_{k\ge 1}$ are i.i.d. standard normal random variables.
Then
\[
S^{(N)}(t) - S^{(N)}(0)
= \sum_{t_{j}<t}
   \theta^{(N)}(t_{j-1})\,\xi^{(N)}_j
 + \sigma\sum_{t_{j}<t}\sqrt{\Delta t}\,Z_j.
\]
Using $\xi^{(N)}_j = \dot\zeta^{(N)}_{t_{j-1}}\,\Delta t$, this becomes
\[
S^{(N)}(t)
= S^{(N)}(0)
 + \sum_{t_{j}<t}
   \theta^{(N)}(t_{j-1})\,\dot\zeta^{(N)}_{t_{j-1}}\,\Delta t
 + \sigma W^{(N)}_t,
\]
where $W^{(N)}_t = \sum_{t_{j}<t}\sqrt{\Delta t}\,Z_j$ is a scaled random walk.
Under the uniform convergence $\theta^{(N)}\to\theta$ and suitable boundedness
of $\dot\zeta^{(N)}$, the Riemann sums
\[
\sum_{t_{j}<t}
   \theta^{(N)}(t_{j-1})\,\dot\zeta^{(N)}_{t_{j-1}}\,\Delta t
\]
converge to $\int_0^t \theta_s \dot\zeta_s\,ds$, and by Donsker’s theorem
$W^{(N)}$ convergence to a Brownian motion $W$. Thus, in the limit $N\to\infty$,
$S^{(N)}$ converges in distribution to $S$ solving
\[
dS_t = \theta_t \dot\zeta_t\,dt + \sigma\,dW_t,
\]
which is exactly Eq.~\eqref{equation:_CT_cash_process}.
\medskip\noindent
\textit{Step 4: Continuous-time limit of the cash process and execution price.}
The discrete cash increment on $[t_{k-1},t_k)$ is
\[
X^{(N)}_k - X^{(N)}_{k-1}
= -\xi^{(N)}_k S^{(N)}_{k-1}
  - \eta^{(N)}_k \big(\xi^{(N)}_k\big)^2.
\]
Using $\xi^{(N)}_k = \dot\zeta^{(N)}_{t_{k-1}}\,\Delta t$ and
$\eta^{(N)}_k = \eta^{(N)}(t_{k-1})/\Delta t$, we obtain
\[
X^{(N)}_k - X^{(N)}_{k-1}
= -\dot\zeta^{(N)}_{t_{k-1}}\,\Delta t\, S^{(N)}_{k-1}
  - \eta^{(N)}(t_{k-1})
    \big(\dot\zeta^{(N)}_{t_{k-1}}\big)^2\,\Delta t.
\]
Summing over $k$ up to time $t$ gives
\[
X^{(N)}(t) - X^{(N)}(0)
= -\sum_{t_{j}<t} \dot\zeta^{(N)}_{t_{j-1}} S^{(N)}_{j-1}\,\Delta t
  -\sum_{t_{j}<t} \eta^{(N)}(t_{j-1})
       \big(\dot\zeta^{(N)}_{t_{j-1}}\big)^2\,\Delta t.
\]
As $N\to\infty$, under the uniform convergence of $S^{(N)}$ and $\eta^{(N)}$
and the boundedness of $\dot\zeta^{(N)}$, these Riemann sums converge to
\[
X_t - X_0
= -\int_0^t \dot\zeta_s S_s\,ds
  -\int_0^t \eta_s \dot\zeta_s^2\,ds.
\]
Equivalently, in differential form,
\[
dX_t = -\dot\zeta_t (S_t + \eta_t \dot\zeta_t)\,dt,
\]
which is Eq.~\eqref{equation:_CT_cash_process_2}.
Finally, at the discrete level, the execution price on step $k$ is
\[
\tilde S^{(N)}_k = S^{(N)}_{k-1} + \eta^{(N)}_k \xi^{(N)}_k
= S^{(N)}_{k-1} + \frac{\eta^{(N)}(t_{k-1})}{\Delta t}
                \dot\zeta^{(N)}_{t_{k-1}}\Delta t
= S^{(N)}_{k-1} + \eta^{(N)}(t_{k-1})\dot\zeta^{(N)}_{t_{k-1}},
\]
which converges to $S_t + \eta_t\dot\zeta_t$ in the limit.
\medskip
Thus, starting from the discrete-time Almgren--Chriss model
\eqref{eq:S}--\eqref{eq:S_tilde} with coefficients
$(\theta^{(N)}_k,\eta^{(N)}_k)$, and under the natural scaling
\[
\theta^{(N)}(t) = \theta^{(N)}_k,
\qquad
\eta^{(N)}(t) = \Delta t\,\eta^{(N)}_k,
\quad
t\in[t_{k-1},t_k),
\]
we obtain in the limit $N\to\infty$ the continuous-time model
\eqref{equation:_CT_cash_process}--\eqref{equation:_CT_cash_process_2}.
\begin{remark}
The scaling $\eta^{(N)}(t) = \Delta t\,\eta^{(N)}_k$ is chosen so that
\[
\eta^{(N)}_k (\xi^{(N)}_k)^2
= \eta^{(N)}(t_{k-1}) \big(\dot\zeta^{(N)}_{t_{k-1}}\big)^2 \Delta t,
\]
which is a Riemann sum for $\int_0^T \eta_t \dot\zeta_t^2 dt$. In other words,
$\eta^{(N)}_k$ is a discrete impact per \emph{block trade}, whereas
$\eta_t$ is an impact per unit of \emph{trading rate}.
\end{remark}

\subsubsection{Proof of Proposition \ref{proposition:ExistenceandUniquenessCT}}

By definition the Optimal Execution Problem in continuous time
\begin{eqnarray}
\zeta^{\star} = \arg \max_{\zeta\in\Xi(Q)}\mathcal{J}(\zeta)
\end{eqnarray}
We apply the Euler--Lagrange identity: necessarily, an optimal execution
strategy is an extremal point (Definition~\ref{def:extremal_point}) and must
solve the boundary-value problem \eqref{eq:extremal_point_BVP}.
This is not a Cauchy problem, and existence and uniqueness of such a boundary value problem is not trivial. 
\begin{lemma}
\label{lemma:existenceanduniquenesstocontrolproblems}
(see \cite{keller2018numerical}). Let $y\in {\mathbb R}^n$, $f: (t,u_1,u_2)\in [0,T]\times\mathbb{R}^{n}\times\mathbb{R}^{n}\rightarrow \mathbb{R}^{n}$ and consider the Boundary Value Problem (BVP-ODE)
    \begin{align}
        \begin{cases}
            \ddot y  &= f(t,y,\dot y)\\
            A_0y(a)&=\alpha\\
            B_0y(b)&=\beta
        \end{cases}
    \end{align}
    with $a \le t \le b$.  Assume $f$ is continuous, $\frac{\partial f_{i}(t,u)}{\partial u_{j}}$ is continuous, $\lVert\frac{\partial f_{i}(t,u)}{\partial u_{j}}\rVert_{+\infty}<k(t)$ such that $\int_{a}^{b}k(x)dx\le \log(1+\frac{\lambda}{m})$, where $m=\lVert (A_0+B_0)^{-1}B_0\rVert_{\infty}$, and $A_0+B_0\in \mathbb{GL}_{n}(\mathbb{R})$ and $\lambda \in(0,1)$. Then the BVP-ODE problem has a unique solution.
\end{lemma}
%    Lemma \ref{lemma:existenceanduniquenesstocontrolproblems}'s proof can be found in"Numerical Methods for Two-Point Boundary-Value Problems" by Keller \cite{keller2018numerical}.

We apply Lemma~\ref{lemma:existenceanduniquenesstocontrolproblems} to the BVP
\eqref{eq:extremal_point_BVP} with $a=0$, $b=T$ and
$f(t,u_{1},u_{2}) = \Gamma_{t}u_{2} + \gamma_{t}u_{1}$, where
$\gamma_{t}=\frac{\dot \theta_{t}}{2\eta_{t}}$ and
$\Gamma_{t} = -\frac{\dot \eta_{t}}{\eta_{t}}$.
Note also that we have $A_0=B_0=1$ and $\alpha=Q$ and $\beta=0$; $A_0$ and $B_0$
are non-zero, and thus invertible, and $m = \frac{1}{2}$.
We look for the least restrictive $k$, i.e. the one implying the weakest constraints on the impact functions, such that 
\begin{eqnarray}
    \int_{0}^{T}k(x)dx\le \log(1+2\lambda)
\end{eqnarray}
and 
\begin{eqnarray}
     \max\{\sup_{0\le s \le t}\{|\gamma_{s}|\},\sup_{0\le s \le t}\{|\Gamma_{s}|\}\} \le k(t)
\end{eqnarray} 
This supremum is attained when $\eta$ and $\theta$ are continuously differentiable, with $\eta$ positive, since $[0,T]$ is compact. Hence a sufficient condition for existence and uniqueness is
\begin{align}\label{eq:condition}
    \begin{cases}
        \frac{1}{2}\int_{0}^{T}\left| \frac{\dot \theta_{t}}{\eta_{t}}\right| dt &\le \log(1+2\lambda) \\ 
        \int_{0}^{T}\left| \frac{\dot \eta_{t}}{\eta_{t}}\right|dt &\le \log(1+2\lambda)
    \end{cases}
\end{align}
for some $\lambda \in (0,1)$. 

\subsubsection{Proof of Proposition \ref{proposition:loclaminimumCalculusofVariations}}

\label{proof:proposition:loclaminimumCalculusofVariations_proof}
We perform a second order expansion of $\mathcal{J}$ around the extremal point
in the direction of $\phi\in \mathcal{C}^{\infty}_{c}((0,T))$, and use the
chain rule. We obtain 
\begin{equation}
\label{eq:_def_od_delta_C}
\mathcal{J}(\zeta+\epsilon\phi)-\mathcal{J}(\zeta)
= \frac{\epsilon^{2}}{2}\underbrace{\int_{0}^{T}\left((\phi)^{2}\partial^{2}_{\zeta,\zeta}L(t,\zeta,\dot \zeta)+2(\phi\dot \phi)\partial^{2}_{\zeta,\dot \zeta}L(t,\zeta,\dot \zeta)+(\dot \phi)^{2}\partial^{2}_{\dot \zeta,\dot \zeta}L(t,\zeta,\dot \zeta)\right)dt}_{=:\delta^{2}\mathcal{J}(\phi,\zeta)}+O(\epsilon^{3}).
\end{equation}
where $L(t,\zeta,\dot \zeta) = \theta_t\zeta_{t}\dot \zeta_t-\eta_t\dot \zeta_t^2$ and $\epsilon>0$. $\zeta$ is a local maximum, and thus necessarily
\begin{eqnarray}\label{eq:min}     \int_{0}^{T}\left((\phi)^{2}\partial^{2}_{\zeta,\zeta}L(t,\zeta,\dot \zeta)+2(\phi\dot \phi)\partial^{2}_{\zeta,\dot \zeta}L(t,\zeta,\dot \zeta)+(\dot \phi)^{2}\partial^{2}_{\dot \zeta,\dot \zeta}L(t,\zeta,\dot \zeta)\right)dt \le 0
\end{eqnarray}
or equivalently 
\begin{eqnarray}
    \int_{0}^{T}\phi(t)\dot \phi(t)\theta_{t}dt \le \int_{0}^{T}(\dot \phi(t))^{2}\eta_{t}dt\text{,}~~~~~~\phi \in \mathcal{C}^{\infty}_{c}((0,T)).
\end{eqnarray}

\subsubsection{Sturm--Picone lemma}

The main proposition providing sufficient conditions for maximum makes use of the following well-known Sturm-Picone theorem \cite{Leetman_Sturm_Picone}.

\begin{lemma}[Sturm--Picone]
\label{lemma:SturmPicone}
Let $p_{i},q_{i}$ with $i\in\{1,2\}$ be real-valued functions on $[a,b]$ such that 
\begin{align}
    \begin{cases}
        (p_{1}(x)y^{\prime})^{\prime}+q_{1}(x)y &= 0 \\
        (p_{2}(x)y^{\prime})^{\prime}+q_{2}(x)y &= 0 
    \end{cases}
\end{align}
and assume that
\begin{align}
    \begin{cases}
        0 < p_{2}(x) \le p_{1}(x) \\
        q_{1}(x) \le q_{2}(x)
    \end{cases}
\end{align}
Let $u$ be a non-trivial solution to the first equation, and $v$ to the second equation. If $z_{1}$ and $z_{2}$ are two successive roots of $u$, then one of the properties holds:
\begin{itemize}
    \item There exists an $x\in(z_{1},z_{2})$ such that $v(x)=0$.
    \item There exists $\lambda \in \mathbb{R}$ such that $u(x) = \lambda v(x)$.
\end{itemize}
\end{lemma}

\subsubsection{Proof of Proposition \ref{proposition:CT_is_well_defined}}
\label{proof:proposition:CT_is_well_defined}
Let $\zeta$ be an extremal point of $\mathcal{J}$ and let
$\phi\in \mathcal{C}^{\infty}_{c}((0,T))$. For
$L(t,\zeta,\dot\zeta)=\theta_t\zeta_t\dot\zeta_t-\eta_t\dot\zeta_t^{2}$ we have
\[
\partial^{2}_{\zeta,\zeta}L=0,\qquad
\partial^{2}_{\zeta,\dot\zeta}L=\theta_t,\qquad
\partial^{2}_{\dot\zeta,\dot\zeta}L=-2\eta_t.
\]
Plugging these into \eqref{eq:_def_od_delta_C} yields
\[
\delta^{2}\mathcal{J}(\phi,\zeta)
=\int_{0}^{T}\bigl(2\theta_t\,\phi(t)\dot\phi(t)-2\eta_t(\dot\phi(t))^{2}\bigr)\,dt.
\]
Since $\phi$ vanishes near $0$ and $T$, an integration by parts gives
\[
\int_{0}^{T}\theta_t\,\phi(t)\dot\phi(t)\,dt
=\frac12\int_{0}^{T}\theta_t(\phi(t)^{2})'\,dt
=-\frac12\int_{0}^{T}\dot\theta_t\,\phi(t)^{2}\,dt.
\]
Therefore,
\[
\delta^{2}\mathcal{J}(\phi,\zeta)
=-\int_{0}^{T}\Bigl(\dot\theta_t\,\phi(t)^{2}+2\eta_t(\dot\phi(t))^{2}\Bigr)\,dt.
\]
Let $\bar\eta=\min_{t\in[0,T]}\eta_t>0$ and
$\bar\theta=\max_{t\in[0,T]}\max\{-\dot\theta_t,0\}$. Then
$-\dot\theta_t\le \bar\theta$ for all $t\in[0,T]$, and
\[
\delta^{2}\mathcal{J}(\phi,\zeta)
\le
\int_{0}^{T}\Bigl(\bar\theta\,\phi(t)^{2}-2\bar\eta(\dot\phi(t))^{2}\Bigr)\,dt.
\]
By the Poincar\'e inequality on $(0,T)$ (using that $\phi(0)=\phi(T)=0$),
\[
\int_{0}^{T}\phi(t)^{2}\,dt \le \frac{T^{2}}{\pi^{2}}\int_{0}^{T}(\dot\phi(t))^{2}\,dt.
\]
Hence
\[
\delta^{2}\mathcal{J}(\phi,\zeta)
\le
\left(\frac{\bar\theta T^{2}}{\pi^{2}}-2\bar\eta\right)\int_{0}^{T}(\dot\phi(t))^{2}\,dt.
\]
If $\sqrt{\gamma}T<\pi$ with $\gamma=\frac{\bar\theta}{2\bar\eta}$, then
$\frac{\bar\theta T^{2}}{\pi^{2}}-2\bar\eta<0$, and therefore
$\delta^{2}\mathcal{J}(\phi,\zeta)<0$ for every nonzero
$\phi\in \mathcal{C}^{\infty}_{c}((0,T))$. Thus $\zeta$ is a strict local
maximum of $\mathcal{J}$. Under the assumed uniqueness of the extremal point
solving \eqref{eq:extremal_point_BVP}, this local maximum is the unique global
maximum, so the Optimal Execution Problem is a well-defined maximization
problem.

\subsubsection{Proof of Proposition \ref{proposition:boundedbelow}}

Write the extremal equation \eqref{eq:extremal_point_BVP} as
\[
\bigl(2\eta_t\dot\zeta_t\bigr)'+(-\dot\theta_t)\zeta_t=0,
\qquad t\in(0,T).
\]
Define $p_1(t)=2\eta_t$ and $q_1(t)=-\dot\theta_t$. By assumption,
$p_1(t)\ge 2\bar\eta>0$ and $q_1(t)\le\bar\theta$ on $[0,T]$.

Consider the comparison equation
\[
\bigl(2\bar\eta\,\dot v_t\bigr)'+\bar\theta\,v_t=0,
\qquad v(0)=Q,\ v(T)=0,
\]
whose solution is
\[
v(t)=Q\frac{\sin\!\bigl(\sqrt{\gamma}\,(T-t)\bigr)}{\sin(\sqrt{\gamma}\,T)},
\qquad \gamma=\frac{\bar\theta}{2\bar\eta}.
\]
If $\sqrt{\gamma}\,T<\pi$, then $\sin(\sqrt{\gamma}\,T)\neq 0$ and $v(t)>0$ for
all $t\in[0,T)$.

Since $0<p_2:=2\bar\eta\le p_1(t)$ and $q_1(t)\le q_2:=\bar\theta$ on $[0,T]$,
we can apply the Sturm--Picone lemma (Lemma~\ref{lemma:SturmPicone}) with $u=\zeta$
and $v$ as above. If $\zeta$ had a zero $t_0\in(0,T)$, then $t_0$ and $T$ would
be successive zeros of $\zeta$, and Lemma~\ref{lemma:SturmPicone} would imply
that $v$ has a zero in $(t_0,T)$ or that $\zeta$ and $v$ are proportional. Both
cases are impossible because $v$ has no zero in $(0,T)$. Therefore $\zeta$ has
no zero in $(0,T)$, and since $\zeta(0)=Q>0$ we conclude $\zeta_t>0$ for all
$t\in[0,T)$.

\subsubsection{Proof of Proposition \ref{proposition:IncreasingImpactsCT}}

When permanent impact is increasing, $\gamma=0$ by definition, and thus by Proposition \ref{proposition:boundedbelow} the solution is non-negative. Since $\gamma=0$, the optimal solution is non-negative on $[0,T]$. We prove by contradiction that the optimal solution cannot be locally strictly concave. By continuity this eliminates all the non-monotonic and non-convex solutions. We assume that the solution to \ref{eq:extremal_point_BVP} is locally strictly concave, i.e. there exist $t_{a}<t_{b}$ in $[0,T]$ such that $\ddot \zeta_{t} < 0$ for all $t\in(t_{a},t_{b})$. By the boundary condition and continuity of $\zeta$ necessarily there exists $t_{c}\in(t_{a},t_{b})$ such that $\dot \zeta_{t_{c}} < 0$. We recall that $\ddot \zeta_{t} = -\frac{\dot \eta_{t}}{\eta_{t}}\dot \zeta_{t}+\frac{\dot \theta_{t}}{2\eta_{t}}\zeta_{t}$, where $\zeta_{t}\ge0$. We assumed that $\dot \eta_{t} \ge0$ and $\dot \theta_{t}\ge0$. Hence $\ddot \zeta_{t_{c}} > 0$. This contradicts the local strict concavity. Hence when impacts are increasing, the optimal solution is convex and monotonic. Finally it is upper bounded by the TWAP inventory by the convexity property. 

\subsubsection{Proof of Proposition \ref{proposition:ConstantTemporaryDecreasingPermanentCT}}

The optimal solution satisfies 
\begin{eqnarray}
    \ddot \zeta_{t} = \frac{\dot \theta_{t}}{2\eta}\zeta_{t}
\end{eqnarray}
We prove by contradiction that the optimal solution cannot be locally strictly convex. Assume by contradiction that $\zeta_t$ is locally strictly convex. Hence there exist $t_{a}<t_{b}$ in $[0,T]$ such that $\ddot \zeta_{t} > 0$ for all $t\in(t_{a},t_{b})$. By proposition \ref{proposition:boundedbelow} the optimal solution is non-negative. We assumed that $\dot \theta_{t}\le0$. Hence $\ddot \zeta_{t} \le 0$ on $(t_{a},t_{b})$. This contradicts the local strict convexity. Finally by the concave property, the optimal solution $\zeta$ is lower bounded by the TWAP.

\subsubsection{Proof of Proposition \ref{proposition:Temporary_decrease_Permanent_increase}}

Since $\theta$ is increasing we have $\bar\theta=0$ and hence $\gamma=0$. By
Proposition~\ref{proposition:boundedbelow}, the solution of
\eqref{eq:extremal_point_BVP} is non-negative on $[0,T]$.

Assume by contradiction that the optimal inventory is not monotone, i.e., there
exists $t_{0}\in(0,T)$ such that $\dot \zeta_{t_{0}}>0$ (a buy period). Then by
continuity there exists $t_{1}\in(t_{0},T)$ such that $\zeta_{t_{1}}>\zeta_{t_{0}}$.
Since $\zeta_{T}=0$, the continuous function $\zeta$ attains its maximum on
$[t_{0},T]$ at some $t^{\star}\in(t_{0},T)$. At this point,
\[
\dot \zeta_{t^{\star}}=0,\qquad \ddot \zeta_{t^{\star}}\le 0,\qquad
\zeta_{t^{\star}}>0.
\]
Evaluating \eqref{eq:extremal_point_BVP} at $t^{\star}$ and using
$\dot\zeta_{t^{\star}}=0$ gives
\[
\ddot \zeta_{t^{\star}}
=\frac{\dot\theta_{t^{\star}}}{2\eta_{t^{\star}}}\zeta_{t^{\star}}.
\]
Under the assumption that $\theta$ is strictly increasing we have
$\dot\theta_{t^{\star}}>0$ and $\eta_{t^{\star}}>0$, hence the right-hand side is
strictly positive. This contradicts $\ddot \zeta_{t^{\star}}\le 0$. Therefore
$\dot \zeta_{t}\le 0$ for all $t\in[0,T]$, i.e., the optimal strategy involves
only sell trades. Hence the model does not admit transaction-triggered price
manipulation, and by Proposition~\ref{proposition:NoPTPimpliesNOPM} it does not
admit price manipulation either.

\subsubsection{Proof of Proposition \ref{proposition:DecreasingImpacts}}

By Proposition~\ref{proposition:boundedbelow}, the optimal solution is
non-negative on $[0,T]$.

Moreover, the optimal solution cannot satisfy $\dot \zeta_{t} < 0$ on some
interval and later become increasing. Indeed, assume by contradiction that
there exist $t_{1},t_{2}$ and $t_{3}$ in $[0,T]$ such that $\dot \zeta_{t} < 0$ on
$(t_{1},t_{2})$, $\dot \zeta_{t_{2}}=0$, and $\dot \zeta_{t} > 0$ on $(t_{2},t_{3})$,
with $\zeta_{t_{1}},\zeta_{t_{2}},\zeta_{t_{3}}>0$. Since both impacts are
decreasing, $\zeta$ satisfies
\[
\ddot \zeta_{t}
=\frac{|\dot \eta_{t}|}{\eta_{t}}\dot \zeta_{t}
-\frac{|\dot \theta_{t}|}{2\eta_{t}}\zeta_{t}.
\]
Hence on $(t_{1},t_{2})$ we have $\dot \zeta_t<0$ and $\zeta_t>0$, so
\[
\ddot \zeta_{t}
=-\frac{|\dot \eta_{t}|}{\eta_{t}}|\dot \zeta_{t}|
-\frac{|\dot \theta_{t}|}{2\eta_{t}}\zeta_{t}<0.
\]
This contradicts the fact that $\dot\zeta$ increases from negative values to
$0$ on $(t_{1},t_{2})$.

Hence if the optimal solution is non-increasing at $t_{0}\in[0,T]$, then it must be non-increasing on $[t_{0},T]$.
This implies that if the initial velocity of $\zeta$ is negative, then $\zeta$ is decreasing. 

Note that an initial positive velocity is possible by concavity arguments. In this framework, by the final boundary condition and continuity, necessarily there exists $t_{4}\in[0,T]$ such that $\zeta$ decreases on $[t_{4},T]$.

We have shown that the optimal solution can have two regimes. In the first, it
decreases from time $0$ until time $T$, and thus there is no
transaction-triggered price manipulation; in that case $\zeta$ is concave. In
the second, transaction-triggered price manipulation occurs: initially the
optimal solution buys up to time $t_{4}$ (the turning point), and then it sells
until time $T$.

Moreover, note that the way it decreases (in both regimes) is never convex, hence the optimal solution is bounded from below by the TWAP. 

Note that the assumption $\sqrt{\gamma}T<\pi$ was sufficient to prevent
catastrophic oscillations. Indeed, the optimal solution cannot be locally convex
when it is positive, but it can when it is negative in principle.

\subsubsection{Proof of Proposition \ref{proposition:Temporary_increase_Permanent_decrease}}

Since $\sqrt{\gamma}T<\pi$, Proposition~\ref{proposition:boundedbelow} yields
$\zeta_t\ge 0$ for all $t\in[0,T]$.

Rewrite \eqref{eq:extremal_point_BVP} in divergence form as
\[
\bigl(2\eta_t\dot\zeta_t\bigr)'=\dot\theta_t\,\zeta_t.
\]
Since permanent impact is decreasing, $\dot\theta_t\le 0$ on $[0,T]$, and hence
$\dot\theta_t\,\zeta_t\le 0$ because $\zeta_t\ge 0$. Therefore the map
$t\mapsto 2\eta_t\dot\zeta_t$ is non-increasing on $[0,T]$. As $\eta_t>0$ for
all $t$, the trading rate $\dot\zeta_t$ can change sign at most once, and if it
does it can only change from positive to negative. Since $\zeta(0)=Q>0$ and
$\zeta(T)=0$, the optimal solution is either monotone decreasing, or it buys
initially and then switches once to selling (one turning point).

\section{Power law decreasing impacts}\label{app:powerlaw}

We assume a decreasing monotonic form for the permanent impacts. 
\begin{align}
    \theta_{t} = \beta_{p}\left(1+\frac{t}{T}\right)^{-\alpha_{p}} && \eta_{t} = \beta_{tp}\left(1+\frac{t}{T}\right)^{-\alpha_{tp}}
\end{align}
This decay might for example be observed at the beginning of the trading day when liquidity decreases \cite{cartea2016incorporating}.
With the above choice it is $-\dot \theta_{t} \le \frac{\alpha_{p}\beta_{p}}{T}=\bar \theta$ and %$\eta_{t} \ge \beta_{tp}(1+\frac{t}{T})^{-\alpha_{tp}}\ge \beta_{tp}2^{-\alpha_{tp}}$ {\bf I would simply say 
$\bar \eta=\beta_{tp}2^{-\alpha_{tp}}$
%} for all $t\in[0,T]$. 
and hence, by proposition \ref{proposition:CT_is_well_defined}, the Optimal Execution Problem is a maximization problem when 
\begin{eqnarray}
    \sqrt{\frac{\alpha_{p}\beta_{p}}{2T\beta_{tp}2^{-\alpha_{tp}}}}T < \pi
\end{eqnarray}
since $\gamma = \frac{\bar{\theta}}{2\bar{\eta}}$ where $\bar{\theta} = \frac{\alpha_{p}\beta_{p}}{T}$ and $\bar{\eta}=\beta_{tp}2^{-\alpha_{tp}}$.
By dimensionality argument $\beta_{tp}=T\kappa\beta_{p}$, with $\kappa$ a dimensionless quantity%Denoting by $\kappa$ the proportionality coefficient that links $\beta_{p}$ and $T\beta_{tp}$, ie $\beta_{tp}=T\kappa\beta_{p}$, 
and the maximization inequality constraint can be written as 
\begin{eqnarray}
\label{inequality:_PL_diff_impacts_maximization}
    \alpha_{p} < \pi^{2}\kappa2^{1-\alpha_{tp}}
\end{eqnarray}
We can give a first interpretation of this results. Indeed, the decreasing velocity of the permanent impact is bounded by a constant that is proportional to the ratio $\frac{\kappa}{T}$ and is a decreasing function of the velocity of the temporary impact. As temporary impact gets small, the permanent impact is lower bounded.

We now consider separately the two cases when the two exponent are different or equal and, without loss of generality, we set $T=1$.
\begin{itemize}
    \item\underline{$\alpha_{tp} \neq \alpha_{p}$}:
    The first order condition is
%\begin{eqnarray}
 %   \ddot \zeta_{t} = \frac{\alpha_{tp}}{T}\frac{1}{1+t/T}\dot \zeta_{t} - \frac{\alpha_{p}}{2\kappa}\frac{1}{(1+t/T)^{\alpha_{p}-\alpha_{tp}+1}}\zeta_{t}
%\end{eqnarray}
%with $\zeta_{0} = Q$ and $\zeta_{T} = 0$. Without loss of generality we can assume $T=1$, and we get
\begin{eqnarray}
    \ddot \zeta_{t} = \frac{\alpha_{tp}}{1+t}\dot \zeta_{t} - \frac{\alpha_{p}}{2\kappa}\frac{1}{(1+t)^{\alpha_{p}-\alpha_{tp}+1}}\zeta_{t}
\end{eqnarray}
with $\zeta_{0} = Q$ and $\zeta_{T} = 0$.
A sufficient condition for this BVP to have a solution is that there exists $\lambda \in (0,1)$ such that 
\begin{align}
\label{conditions:PL_different_exponents}
        \begin{cases}
            \frac{\alpha_{p}}{2\kappa(\alpha_{tp}-\alpha_{p})}\left[1-2^{\alpha_{tp}-\alpha_{p}}\right] &\le \log(1+2\lambda) \\ 
            \alpha_{tp}\log(2) &\le \log(1+2\lambda)
        \end{cases}
\end{align}

\begin{figure}
        \centering
        \includegraphics[scale=0.37]{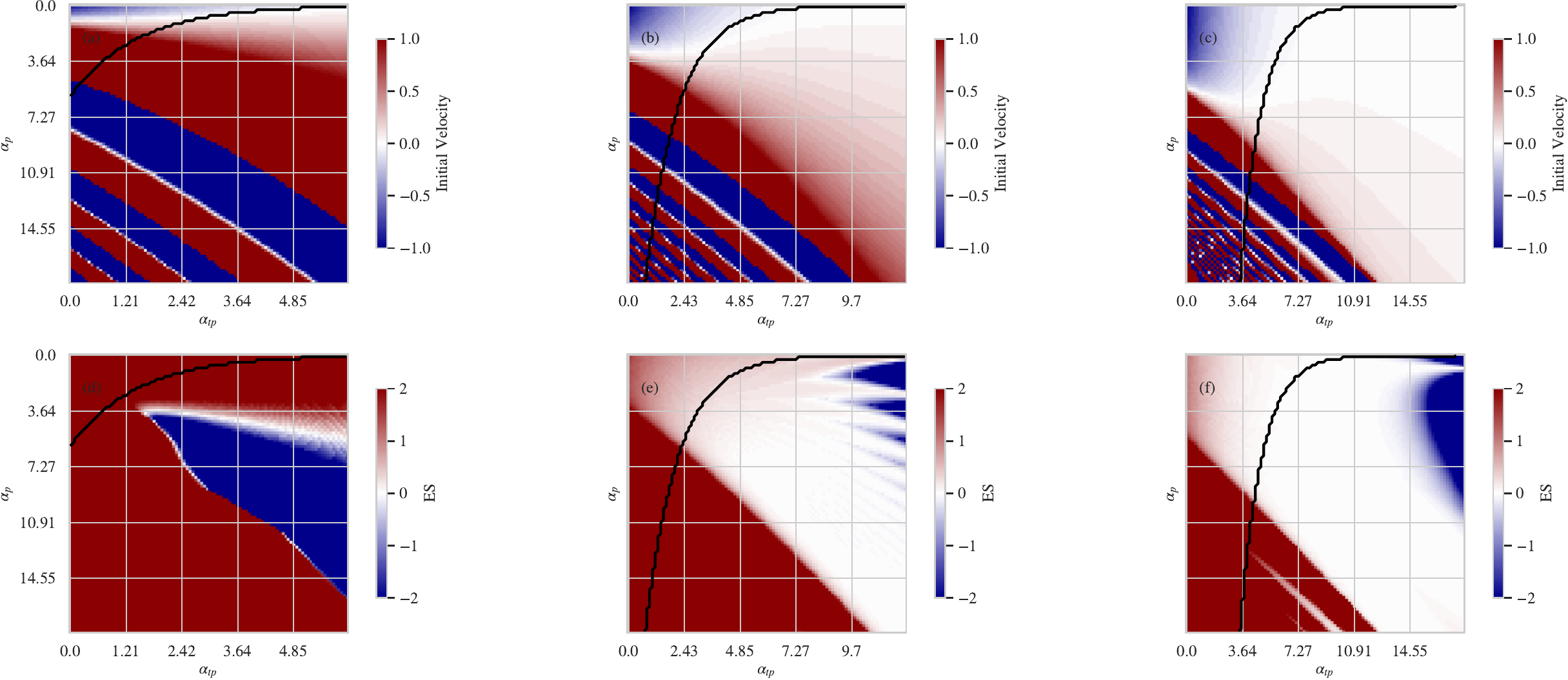}
        \caption{ \textbf{First row:} Comparison of the initial velocities with respect to different market conditions, and admissible regimes. \textbf{Second row:} Cash discrepancy associated to the optimal strategies. \textbf{Caption:} Each column is associated to a parameter $\kappa$ (first column with $\kappa=0.1$, second column with $\kappa=1$, third column with $\kappa=10$). The region on the left  of the black line is the admissible region.}\label{fig:init_velocities_diff_PL}
\end{figure}

\bigbreak

To reveal the impact settings that give a maximization problem, we propose to explore the parameter space by examining the initial trading velocities associated with optimal execution strategies across various impact settings. Doing so, we are interested in the cost values as a function of the parameters characterizing the impacts decay. 

The first row of Figure \ref{fig:init_velocities_diff_PL} shows the initial trading velocities as a function of $\alpha_{tp}$ and $\alpha_{p}$, while the second row shows the expected shortfall. The first column considers $\kappa=0.1$. In the admissible region, defined by Eq. \ref{inequality:_PL_diff_impacts_maximization} and corresponding to the region on the left of the black line, we observe initial trading velocities greater than $-1$, corresponding to concave form, in agreement with the proposition \ref{proposition:DecreasingImpacts}. The admissible region contains monotonic concave solutions, represented by the  red region, and non monotonic concave solutions, associated with transaction-triggered manipulation strategies, represented by the blue region. We observe that the transition between the two regions is smooth and seems to be roughly independent of $\alpha_{tp}$. 
The second's row first panel shows the expected shortfall generated by the optimal trading strategies. In the admissible region, this quantity is positive, indicating positive transaction costs. Hence with small values of $\kappa$, our model does not generate price manipulations, but admit non-monotonic solutions. 

The second and third columns show the case $\kappa=1$ and $\kappa=10$, respectively. In the admissible region the  area corresponding to monotonic solutions (blue) becomes larger and larger. 
Interestingly, as observed for $\kappa=0.1$, when $\kappa=1$ and $\kappa=10$, the admissible strategies does not generate negative transaction costs (second row panels). This is somewhat expected, given that when $kappa$ is large there is less space for transaction-triggered price manipulation and thus 
%with relatively large values of $\kappa$, the admissible region only for more and more monotonic solutions, indicating no transaction-triggered price manipulation and, consequently, 
positive transaction costs by Proposition \ref{proposition:NoPTPimpliesNOPM} (see \cite{alfonsi2010optimal}).

In conclusion, the time-varying Almgren-Chriss impact model in continuous time, characterized by power law decreasing impacts with different speed of decrease, can not result in negative transaction costs even if transaction-triggered price manipulation is possible. Moreover, as the value of $\kappa$ increases, indicating $\frac{\eta_{0}}{T}>>\theta_{0}$, the model diminish potential form of transaction-triggered price manipulation, ensuring the absence of negative transaction costs.

\item \underline{$\alpha_{tp} = \alpha_{p}=:\alpha$}:
The first order conditions gives
%\begin{eqnarray}
%    \ddot \zeta_{t} = \frac{\alpha}{T}\frac{1}{1+t/T}\dot \zeta_{t} - \frac{\alpha\beta_{p}}{2T\beta_{tp}}\frac{1}{(1+t/T)}\zeta_{t}
%\end{eqnarray}
%with $\zeta_{0} = Q$ and $\zeta_{T} = 0$. Without loss of generality we can assume $T=1$, and we get
\begin{eqnarray}
    \ddot \zeta_{t} = \frac{\alpha}{1+t}\dot \zeta_{t} - \frac{\alpha}{2\kappa}\frac{1}{(1+t)}\zeta_{t}
\end{eqnarray}
with $\zeta_{0} = Q$ and $\zeta_{T} = 0$.
If there exists $\lambda \in (0,1)$ such that Ineq. \ref{inequalities:_Power_law_examples} are verified, then this BVP has a solution. 
\begin{align}
\label{inequalities:_Power_law_examples}
        \begin{cases}
            \frac{\alpha}{2\log(1+2\lambda)}\log(2) &\le  \kappa \\ 
            \alpha\log(2) &\le \log(1+2\lambda)
        \end{cases}
\end{align}

We can rewrite the maximization inequality \ref{inequality:_PL_diff_impacts_maximization} as
\begin{eqnarray}
\label{inequality:_PL_same_impacts_maximization}
    \alpha < W\left(\frac{2\pi^{2}\kappa}{\log(2)}\right)
\end{eqnarray}
where $W(\cdot)$ is the product logarithm function, called the Lambert function. 

%In this context we have a solution with parameters $\kappa$ and $\alpha$ that can be easily computed numerically with the shooting method \cite{osborne1969shooting}. 

Having only two parameters $\kappa$ and $\alpha$ allows us to completely characterize the impact settings that give any type of manipulation. To explore these impact regimes, we show a heatmap with \(\alpha\) on the x-axis and \(\ \kappa\) on the y-axis, emphasizing the regions of phase space where transaction-triggered price manipulation and manipulation occur. Figure \ref{fig:init_velocities_equal_PL} $(a)$ illustrates the initial trading velocities for optimal execution strategies as a function of  $\kappa$ and $\alpha$. The  region to the left of the black line identifies the admissible market regimes satisfying the inequality (\ref{inequality:same_expo_exponents}). The admissible region defined by Proposition \ref{proposition:CT_is_well_defined} generates optimal solutions that are small perturbations of the TWAP (light blue regions where the slope is close to $-1$) and monotonic and more concave strategies (light blue regions) at the frontiers between the blue and the red regions. In this admissible region, the associated impact settings produce monotonic optimal execution strategies, and by Proposition \ref{proposition:NoPTPimpliesNOPM} there are no negative transaction costs. This property is even more evident in Figure \ref{fig:init_velocities_equal_PL}(b), where the expected shortfall is positive.

Furthermore, observing again panel $(a)$ the admissible region accepts also red strategies, corresponding to concave strategies with positive initial trading velocity. As anticipated by Proposition \ref{proposition:DecreasingImpacts}, the solution is lower bounded by the TWAP, but optimal solution can be monotonic or can have a turning point; this is what we see here. These strategies exhibit a unique local maximum, involving buying in the initial phase of execution and selling thereafter. 

While Proposition \ref{proposition:NoPTPimpliesNOPM} ensures the absence of negative transaction costs in the absence of transaction-triggered price manipulation, the question remains whether such strategies (the one with TTPM) generate negative transaction costs. Plot $(b)$ in Figure \ref{fig:init_velocities_equal_PL} reveals that despite allowing buys in a selling process, these strategies always have positive transaction costs since they generate negative cost.
\end{itemize}
\begin{figure}
        \centering
        \includegraphics[scale=0.45]{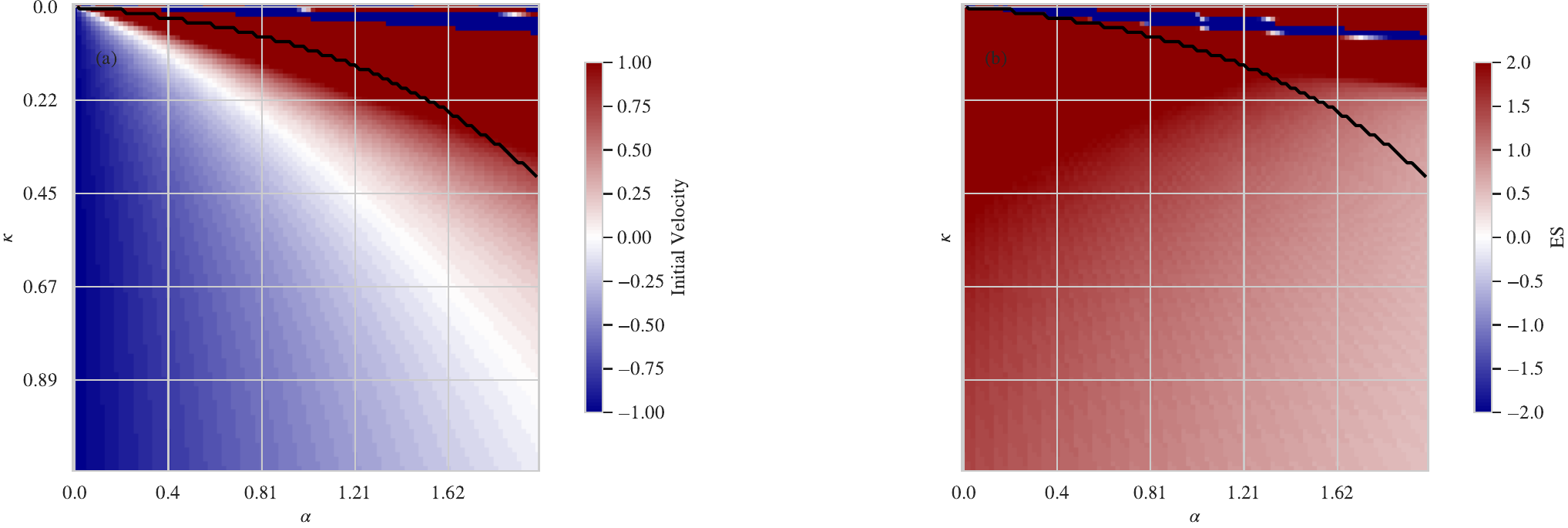}
        \caption{\textit{(a)} Comparison of the initial velocities with respect to different market conditions, and admissible regimes. \textit{(b)} Cash discrepancy associated to the optimal strategies. \textit{Both figures:} The region on the left  of the black line is the admissible region}\label{fig:init_velocities_equal_PL}
\end{figure}

\section{B-matrices}
\label{app:b_matrices}
In this section, we introduce a class of matrices that are a subset of the class of P-matrices: the B-matrices. B-matrices are very useful in matrix analysis and its application. Notably properties of B-matrices are used to localize the real eigenvalues of a real matrix and the real parts of all eigenvalues of a real matrix. Moreover symmetric B-matrices are extremely useful because they are positive definite. 

\begin{definition}
\label{definition:Bmatrix}
    Let $A=(a_{i,k})_{1\le i,k\le n}$ be a real square matrix. A is a B-matrix when it has positive row sums and all of its off-diagonal elements bounded above by the corresponding row means, ie for all $i=1,\dots,n$
    \begin{eqnarray}
        \sum_{k=1}^{n}a_{i,k}>0 \text{, and } \frac{1}{n}\left(\sum_{k=1}^{n}a_{i,k}\right)>a_{i,j} \text{, }\forall j \neq i.
    \end{eqnarray}
\end{definition}
B-matrices have many favorable properties, the most relevant ones are detailed in \ref{app:b_matrices}.

\begin{proposition}
\label{proposition:Bmatrixnecessarycond}
    Let $A$ be a B-matrix. Then necessarily 
    \begin{eqnarray}
        a_{i,i} > \max\{0,a_{i,j}|j\neq i\} \text{, } \forall i \in \{1,\dots,n\}.
    \end{eqnarray}
    Moreover, 
    \begin{eqnarray}
        a_{i,l} \le \frac{1}{n}\sum_{k=1}^{n}a_{i,k} \le a_{i,i} \text{, } \forall l\neq i\text{, }\forall i \in \{1,\dots,n\}.
    \end{eqnarray}
\end{proposition}
\begin{proof}{}
    First we show that $a_{i,i} > \max\{0,a_{i,j}|j\neq i\} \text{, } \forall i \in \{1,\dots,n\}$. We fix $i\in\{1,\dots,n\}$ and remark that $\frac{1}{n}\left(\sum_{k=1}^{n}a_{i,k}\right)>a_{i,j}$ implies that $a_{i,i}>a_{i,j}$, $\forall j \in \{1,\dots,n\}$. We prove by contradiction that $a_{i,i}>0$. Assume that $a_{i,i}\le0$. By assumption and using what we proved before we also have $a_{i,j}<a_{i,i}\le 0$, $\forall j \in \{1,\dots,n\}$. Thus we have $\sum_{k=1}^{n}a_{i,k}\le0$. By the B-matrix definition we have $\sum_{k=1}^{n}a_{i,k}>0$. This is a contradiction, hence $a_{i,i}>0$ necessarily. 

    Now we prove that $a_{i,l} \le \frac{1}{n}\sum_{k=1}^{n}a_{i,k} \le a_{i,i} \text{, } \forall l\neq i\text{, }\forall i \in \{1,\dots,n\}$. The left inequality is given by the second inequality of the $B$-matrix definition. The right inequality follows from the fact that $a_{i,i} > \max\{0,a_{i,j}|j\neq i\}$. Indeed, it follows that $\frac{1}{n}\sum_{k=1}^{n}a_{i,k}\le\frac{1}{n}\sum_{k=1}^{n}a_{i,k}\pmb{1}_{\{a_{i,k}\ge0\}}<a_{i,i}.$
    
\end{proof}

\begin{remark}
    For the moment we gave a definition of B-matrices, and necessary conditions for a matrix to be a B-matrix. In the next proposition we give a consistent criteria for being a B-matrix. 
\end{remark}

\begin{proposition}
\label{proposition:Bmatrixequivalence}
    Let $A=(a_{i,k})_{1\le i,k\le n}$ be a real square matrix. Then $A$ is a B-matrix if and only if 
    \begin{eqnarray}
        \frac{1}{n}\sum_{k=1}^{n}a_{i,k} > \max\{0,a_{i,j}|j\neq i\} \text{, } \forall i \in \{1,\dots,n\}.
    \end{eqnarray}
\end{proposition}

\begin{proof}{}
    See Proposition 2.2 in Peña \cite{pena2001class}.
\end{proof}

\begin{lemma}
\label{lemma:BmatrixSPD}
    A symmetric B-matrix is positive definite. 
\end{lemma}

\begin{proof}{}
    See Corollary 2.7 in Peña \cite{pena2001class}. 
\end{proof}

\begin{lemma}
\label{lemma:christensen}
Let $A \in \mathcal{M}_{n}(\mathbb{R})$. We assume that $A$ is non-singular, and $A^{T}$ is  $B$-matrix. Then $A^{-1}\pmb{1} \ge 0$ in the element-wise sense. 
\end{lemma}

\begin{proof}{}
    See Lemma $4$ in Christensen \cite{christensen2019comparative}. 
\end{proof}

\EndAppendix

\end{document}